\documentclass[final]{siamltex}
\usepackage{graphicx}
\usepackage{leftidx}
\usepackage{mathrsfs}
\usepackage{amssymb}
\usepackage{amsmath}
\usepackage{cite}
\usepackage{float}
\usepackage{multirow}
\usepackage{color}

\newtheorem{remark}[theorem]{Remark}

 \usepackage{leftidx}

\title{Superconvergence points of integer and fractional derivatives of special Hermite interpolations and its applications in solving FDEs
\thanks{This work is supported in part by the National Natural Science Foundation of China under grants NSFC 11471031, NSFC 91430216, 11771035, and NSAF U1530401; the US National Science Foundation through grant DMS-1419040.} }

\author{Beichuan Deng\thanks{Department of Mathematics, Wayne State University, Detroit, MI 48202, USA. (\tt beichuan.deng@wayne.edu).}
\and Jiwei Zhang\thanks{Beijing Computational Science Research Center, Beijing 100193, China ({\tt jwzhang@csrc.ac.cn})}.
\and Zhimin Zhang\thanks{Corresponding author. Beijing Computational Science Research Center, Beijing 100193, China ({\tt zmzhang@csrc.ac.cn}); and Department of Mathematics, Wayne State University,  Detroit, MI 48202, USA, (\tt zzhang@math.wayne.edu).}
}

\begin{document}
\maketitle

\begin{abstract}
In this paper, we study  convergence and superconvergence theory of integer and fractional derivatives of the one-point and the two-point Hermite interpolations. When considering the integer-order derivative, exponential decay of the error is proved, and superconvergence points are located, at which the convergence rates are $O(N^{-2})$ and $O(N^{-1.5})$, respectively, better than the global rate for the one-point and two-point interpolations.  Here $N$ represents the degree of interpolation polynomial.
 It is proved that the $\alpha$-th fractional derivative of $(u-u_N)$ with $k<\alpha<k+1$, is bounded by its $(k+1)$-th derivative. Furthermore, the corresponding superconvergence points are predicted for fractional derivatives, and an eigenvalue method is proposed to calculate the superconvergence points for the Riemann-Liouville fractional derivative. In the application of the knowledge of superconvergence points to solve FDEs, we discover that a modified collocation method makes numerical solutions much more accurate than the traditional collocation method.
\end{abstract}



\pagestyle{myheadings}
\thispagestyle{plain}
\markboth{ }{Superconvergence of one/two-point Hermite interpolations}
\section{Introduction}
Nowadays, Lagrange-type interpolation is applied widely in numerical analysis, especially in the spectral Galerkin/collocation methods and numerical quadratures based on (weighted) orthogonal polynomials, because it is stable and converges extremely fast for smooth functions. Recently, many works focus on studying the superconvergence of Lagrange-type interpolations. For example, the superconvergence of integer-order derivative of various Jacobi-Gauss-type spectral interpolations were considered in \cite{Zhang+2008+JSC, Zhang+2012+SINUM,LWang+2014+JSC,Xie+2012+MACOM} respectively. Zhang, Zhao and Deng generalized the study to Riemann-Liouville and Riesz fractional derivatives with the order of $0<\alpha<1$ in \cite{Deng+arxiv, Zhao+2016+SISC}. However, as pointed out in \cite{Deng+arxiv}, Lagrange-type interpolations fail to converge when applying to left Riemann-Liouville fractional derivative with $\alpha>1$. More specifically, the error ${_{-1}D_x^{\alpha}}(u-u_N)$ will not converge at $x=-1$, where $u_N$ is the Legendre-Lobatto or Left-Radau interpolant of $u(x)$. The main reason is that the multiplicity of zero point of $u-u_N$ at $x=-1$ is one.  Similarly, Lagrange-type interpolations also fail at $x=\pm1$ when applying to Riesz fractional derivatives with $\alpha>1$. This motivates us to find an alternative way to remedy the situations where Lagrange-type interpolations fail.

The Hermite interpolation is a natural idea since it can interpolate the function value at interior interpolation points, and interpolate both function value and derivative value(s) at two ends $x=\pm1$. This leads to the zero multiplicity higher at the two ends.  Consequently, the Hermite interpolation makes the singularity at $x=\pm1$ to be under control comparing with Lagrange-type interpolations.  While the theory of convergence and superconvergence for Lagrange-type interpolations is well understood
(see, \cite{Zhang+2008+JSC, Zhang+2012+SINUM,LWang+2014+JSC,Deng+arxiv, Zhao+2016+SISC}), the study on Hermete interpolation has so far received little attention.

The goal of this paper is to provide a fundamental analysis of the convergence and superconvergence for the integer-order and fractional derivatives of three kinds of Hermite interpolants, see the definitions in (\ref{def31}), (\ref{def32}) and (\ref{def33}) below. For the integer-order derivative, it is proved that the error of the Hermite interpolant $u(x)-u_N(x)$ decays exponentially with respect to the degree of interpolation polynomial $N$, when $u(x)$ is analytic on $[-1,1]$. Moreover, for the one-point Hermite interpolation given by (\ref{def31}) or (\ref{def32}), the convergence rate at superconvergence points is $O(N^{-2})$ higher than the optimal global rate; for the two-point Hermite interpolation given by (\ref{def33}), it is $O(N^{-\frac{3}{2}})$ higher. Then, we apply the Hermite interpolation (\ref{def31})/(\ref{def32}) to left/right Riemann-Liouville fractional derivatives, and (\ref{def33}) to Riesz fractional derivative, and a unified error estimate is given by
\begin{eqnarray*}
\Vert {D^{\alpha}}(u-u_{N}) \Vert_{\infty} \leqslant C \Vert (u-u_{N})^{(k+1)}\Vert_{\infty},
\end{eqnarray*}
where $k<\alpha<k+1$, $k$ is a nonnegative integer, and the constant $C$ is independent of $N$.  In other words, the convergence of the interpolation under the fractional derivative is also guaranteed, and the error decays exponentially as well. In order to validate superconvergence phenomenon under the Riemann-Liouville fractional derivatives, an efficient algorithm to calculate superconvergence points is provided, and it is observed that the gain of convergence rate at these points is at least $O(N^{-1})$.

We would like to emphasize that a systematic and rigorous mathematical treatment of
superconvergence points of fractional derivatives can offer some theoretical insight in applications of numerically solving fractional differential equations (FDEs). It may also provide guidance to construct numerical schemes to improve the numerical accuracy.
Note that the fractional derivative of the Hermite interpolant $D^{\alpha}u_N$ approximates $D^{\alpha}u$ more accurately at the superconvergence points. This knowledge enables us to modify the traditional collocation method to solve FDEs: Instead of using the traditional collocation points, the new collocation points are selected as the superconvegence points found in this paper. The existence and uniqueness of numerical solutions can be guaranteed based on the fact that the numbers of interpolation points and superconvergence points are equal to each other for fractional derivatives.
Numerical experiments indicate that the accuracy of the modified collocation method is at least $10^{-2}$ improvement compared with traditional collocation methods.

The paper is organized as follow. Preliminary knowledge is provided in Section 2. Section 3,4,5 are about the theoretical statements of convergence and superconvergence of the Hermite interpolations for taking integer-order derivative, Riemann-Liouville derivative, and Riesz derivative, respectively. Numerical validations and applications are presented in Section 6. Finally, some conclusions are drawn in Section 7.

\section{Preliminaries}
In this section, we begin with some basic definitions and properties, and denote $\mathbb{Z}^+$ and $\mathbb{N}$ by the sets of all positive integers and all nonnegative integers, respectively.

\subsection{Definitions and Properties of Fractional Derivatives}
Let us first recall the definitions and properties of Riemann-Liouville and  Riesz fractional derivatives.

\begin{definition} For $\gamma \in (0,1)$, the left and right Riemann-Liouville fractional integrals are defined respectively by
\begin{eqnarray*}
{_{-1}I_x^{\gamma}}u(x):=\frac{1}{\Gamma(\gamma)} \int_{-1}^x \frac{u(\tau)}{(x-\tau)^{1-\gamma}} d\tau, \ x \in (-1,1], \\
{_{x}I_1^{\gamma}}u(x):=\frac{1}{\Gamma(\gamma)} \int_{x}^1 \frac{u(\tau)}{(\tau-x)^{1-\gamma}} d\tau, \ \ x \in [-1,1).
\end{eqnarray*}
Then for $\alpha \in (k-1,k)$, where $k \in \mathbb{Z}^{+}$, the left and right Riemann-Liouville derivatives are defined respectively by:
\begin{eqnarray}
 _{-1}D_x^{\alpha} u(x) &=& D^k (_{-1}I_x^{k-\alpha}u(x)) \nonumber \\
 &=& \sum_{l=0}^{k-1}\frac{u^{(l)}(-1)}{\Gamma(l+1-\alpha)}(x+1)^{-\alpha+l}+{_{-1}I_x^{k-\alpha}}D^k u(x),  \label{def11} \\
 _{x}D_1^{\alpha} u(x) &=& (-1)^k D^k (_xI_1^{k-\alpha}u(x)) \nonumber \\
&=& \sum_{l=0}^{k-1}(-1)^l\frac{u^{(l)}(-1)}{\Gamma(l+1-\alpha)}(1-x)^{-\alpha+l}+(-1)^k{_xI_1^{k-\alpha}}D^k u(x),
\label{def12}
\end{eqnarray}
where $D^k:=\frac{d^k}{dx^k}$ is the $k$-th  derivative.
\end{definition}

\begin{definition}
Let $\gamma \in (0,1)$, the Riesz potentials in one dimension are defined as follows:
\begin{eqnarray}
&& I^{\gamma}_o u(x):= \frac{c_1}{\Gamma(\gamma)}\int _{-1}^1 \frac{sign(x-\tau) u(\tau)}{|x-\tau|^{1-\gamma}} d\tau = c_1 (_{-1}I_x^{\gamma}-{_x I_1^{\gamma}})u(x), \\
&& I^{\gamma}_e u(x):= \frac{c_2}{\Gamma(\gamma)}\int _{-1}^1 \frac{u(\tau)}{|x-\tau|^{1-\gamma}} d\tau = c_2 (_{-1}I_x^{\gamma}+{_x I_1^{\gamma}})u(x),
\end{eqnarray}
where $sign(\cdot)$ represents the sign function, $c_1=\frac{1}{2 \sin(\pi\gamma/2)}$, $c_2=\frac{1}{2 \cos(\pi\gamma/2)}$. Then for $\alpha \in (k-1,k)$, we can therefore define the Riesz fractional derivative:
\begin{eqnarray}
^R D^{\alpha}u(x):=
\left\{\begin{array}{ll}
D^k I_o^{k-\alpha}u(x)=c_1 (_{-1}D_x^{\alpha} + {_{x}D_1^{\alpha}} )u(x), \ k \ is \ odd,\\
D^k I_e^{k-\alpha}u(x)=c_2 (_{-1}D_x^{\alpha}+ {_{x}D_1^{\alpha}} )u(x), \ k \ is \ even.
\end{array}\right.
\label{def22}
\end{eqnarray}
\end{definition}

\subsection{Generalized Jacobi Polynomials and Their Properties}
We now address the definition of classical Jacobi polynomials and Generalized Jacobi Polynomials (GJP)  with integer indexes. For the convenience of notations, the following four sets are introduced:
\begin{eqnarray*}
& \mathcal{Z}_1=\{ (\alpha,\beta)\in \mathbb{Z}^2: \alpha,\beta<0 \},  \quad
& \mathcal{Z}_2=\{ (\alpha,\beta)\in \mathbb{Z}^2: \alpha<0,\beta \geqslant0 \}, \\
& \mathcal{Z}_3=\{ (\alpha,\beta)\in \mathbb{Z}^2: \alpha \geqslant 0,\beta<0 \},\quad
& \mathcal{Z}_4=\{ (\alpha,\beta)\in \mathbb{Z}^2: \alpha,\beta \geqslant 0 \}. \\
\end{eqnarray*}
\begin{definition}
The classical Jacobi polynomials, denoted by $P^{\alpha,\beta}_n(x)$, are orthogonal with respect to the weight function $\omega^{\alpha,\beta}=(1-x)^{\alpha}(1+x)^{\beta}$ over the interval $I=[-1,1]$, namely,
\begin{eqnarray*}
\int_{-1}^1 P^{\alpha,\beta}_n(x) P^{\alpha,\beta}_m(x) dx=\gamma_n^{\alpha,\beta} \delta_{m,n},
\end{eqnarray*}
where $\delta_{n,m}$ is the Kronecker function, and $\gamma_n^{\alpha,\beta}= \frac{2^{\alpha+\beta+1}\Gamma(n+\alpha+1)\Gamma(n+\beta+1)}{(n+\alpha+\beta+1)\Gamma(n+1)\Gamma(n+\alpha+\beta+1)}$.
\end{definition}
\begin{definition}
Let $\alpha,\beta \in \mathbb{Z}$, the Generalized Jacobi Polynomial $\mathcal{J}^{\alpha,\beta}_N (x)$ is defined as follows:
\begin{eqnarray}
\mathcal{J}^{\alpha,\beta}_N (x)=\left\{\begin{array}{ll}
(1-x)^{-\alpha}(1+x)^{-\beta} P^{-\alpha,-\beta}_{\tilde{n}}(x), & (\alpha,\beta) \in \mathcal{Z}_1, \tilde{n}=N+\alpha+\beta, \\
(1-x)^{-\alpha}P^{-\alpha,\beta}_{\tilde{n}}(x), & (\alpha,\beta) \in \mathcal{Z}_2, \tilde{n}=N+\alpha,  \\
(1+x)^{-\beta}P^{\alpha,-\beta}_{\tilde{n}}(x), & (\alpha,\beta) \in \mathcal{Z}_3, \tilde{n}=N+\beta,\\
P^{\alpha,\beta}_{\tilde{n}}(x), & (\alpha,\beta) \in \mathcal{Z}_4, \tilde{n}=N,
\end{array}\right.
\label{def2.3}
\end{eqnarray}
where $\{ \mathcal{J}^{\alpha,\beta}_N (x)\}$ are only defined for  $N\geqslant -\min\{\alpha,0\}-\min\{\beta,0\}$, and $P^{\alpha,\beta}_n(x)$ is the classical Jacobi polynomial defined above.
\end{definition}

It is easy to check that $\{ \mathcal{J}^{\alpha,\beta}_N (x)\}$ is also a set of weighted orthogonal polynomials. This is, for any $\alpha,\beta \in \mathbb{Z}$, and $N,M\geqslant -\min\{\alpha,0\}-\min\{\beta,0\}$, we have:
\begin{eqnarray}
\int_{-1}^1 \mathcal{J}_N^{\alpha,\beta}(x)\mathcal{J}_M^{\alpha,\beta}(x)\omega^{\alpha,\beta}dx=\eta^{\alpha,\beta}_N \delta_{N,M},
\end{eqnarray}
where $\eta^{\alpha,\beta}_N=\gamma^{|\alpha|,|\beta|}_{\tilde{n}}$,
and $\tilde{n}$ is defined in (\ref{def2.3}). Additionally, when both $\alpha,\beta$ are integers, according to \cite{Guo+Shen+Wang+GJPF, Chen+2014+mathcom,Shen+2015+JMS}, all of the GJPs satisfy the Sturm-Liouville equation:
\begin{eqnarray}
\frac{d}{dx}[(1-x)^{\alpha+1}(1+x)^{\beta+1}\frac{d}{dx}\mathcal{J}^{\alpha,\beta}_N(x)]+\lambda_N^{\alpha,\beta}(1-x)^{\alpha}(1+x)^{\beta}\mathcal{J}_N^{\alpha,\beta}(x)=0,
\label{sturmlv}
\end{eqnarray}
where $\lambda_N^{\alpha,\beta}=N(N+\alpha+\beta+1)$. It is well known that the classical Jacobi polynomials with $\alpha,\beta>-1$ satisfy
\begin{eqnarray}
\frac{d}{dx} P^{\alpha,\beta}_n(x)=\frac{1}{2}(n+\alpha+\beta+1)P^{\alpha+1,\beta+1}_{n-1}(x), \ n\geqslant1.
\end{eqnarray}
As for GJPs, we have similar formulas.
\begin{lemma}[see \cite{Guo+Shen+Wang+GJPF}, Lemma 2.1]
If $(\alpha,\beta) \in \mathcal{Z}_1$, then
\begin{eqnarray}
\frac{d}{dx} \mathcal{J}^{\alpha,\beta}_N(x)=-2(N+\alpha+\beta+1)\mathcal{J}^{\alpha+1,\beta+1}_{N-1}(x);
\label{lemma241}
\end{eqnarray}
if $(\alpha,\beta) \in \mathcal{Z}_2$, then
\begin{eqnarray}
\frac{d}{dx} \mathcal{J}^{\alpha,\beta}_N(x)=-N\mathcal{J}^{\alpha+1,\beta+1}_{N-1}(x);
\label{lemma242}
\end{eqnarray}
if $(\alpha,\beta) \in \mathcal{Z}_3$, then
\begin{eqnarray}
\frac{d}{dx} \mathcal{J}^{\alpha,\beta}_N(x)=N\mathcal{J}^{\alpha+1,\beta+1}_{N-1}(x).
\label{lemma243}
\end{eqnarray}
\label{lemma1}
\end{lemma}
The following theorem states the maximum norm of GJPs.
\begin{theorem}
Let $\mathcal{J}_N^{\alpha,\beta}(x)$ and $\tilde{n}$ be defined in (\ref{def2.3}), where $\alpha,\beta$ are integers. When $\alpha=\beta < 0$, if $\tilde{n}$ is even, then
\begin{eqnarray}
\Vert \mathcal{J}_N^{\alpha,\beta}(x) \Vert_{L^{\infty}[-1,1]}=|\mathcal{J}_N^{\alpha,\beta}(0)|=|P^{-\alpha,-\beta}_{\tilde{n}}(0)|=2^{-\tilde{n}} \binom{\tilde{n}-\alpha}{\tilde{n}/2},
\end{eqnarray}
if $\tilde{n}$ is odd, then
\begin{eqnarray}
\Vert \mathcal{J}_N^{\alpha,\beta}(x) \Vert_{L^{\infty}[-1,1]}=|\mathcal{J}_N^{\alpha,\beta}(\eta_0)| \leqslant 2^{-\tilde{n}} \frac{\Gamma(\tilde{n}-\alpha+1)}{\Gamma(\tilde{n}/2+1)\Gamma(\tilde{n}/2-\alpha+1)},
\end{eqnarray}
where $\eta_0$ is the zero point of $P^{-\alpha-1,-\alpha-1}_{\tilde{n}+1}(x)$ which is closest to 0; when $\alpha \leqslant 0$, $\beta \geqslant 0$, then
\begin{eqnarray}
\Vert \mathcal{J}_N^{\alpha,\beta}(x) \Vert_{L^{\infty}[-1,1]}=|\mathcal{J}_N^{\alpha,\beta}(-1)|=2^{|\alpha|}\binom{\tilde{n}+\beta}{\tilde{n}};
\end{eqnarray}
when $\alpha \geqslant 0$, $\beta \leqslant 0$, then
\begin{eqnarray}
\Vert \mathcal{J}_N^{\alpha,\beta}(x) \Vert_{L^{\infty}[-1,1]}=|\mathcal{J}_N^{\alpha,\beta}(1)|=2^{|\beta|}\binom{\tilde{n}+\alpha}{\tilde{n}},
\end{eqnarray}
where $\binom{n}{m}:=\frac{n!}{m! (n-m)!}$ for integers $n$ and $m$.
\label{thm1}
\end{theorem}

\begin{proof}
According to Lemma \ref{lemma1}, the maximal value of $\mathcal{J}_N^{\alpha,\beta}(x)$ must be located at one of points in $A=\{-1,1 \} \cup \{ \hbox{zeros of} \ P^{-\alpha-1,-\alpha-1}_{\tilde{n}+1}(x) \}$. Therefore, define
\begin{eqnarray}
f(x)=\lambda_N^{\alpha,\beta} [\mathcal{J}_N^{\alpha,\beta}(x)]^2+(1-x^2)[\frac{d}{dx} \mathcal{J}_N^{\alpha,\beta}(x)]^2, \ x \in [-1,1],
\end{eqnarray}
we can see that $\forall \eta \in A$,
\begin{eqnarray*}
f(\eta)=\lambda_N^{\alpha,\beta} [\mathcal{J}_N^{\alpha,\beta}(\eta)]^2,
\end{eqnarray*}
so the question turns out to find $$\max_{\eta \in A} \{ f(\eta) \}.$$
By simplifying (\ref{sturmlv}), we have:
\begin{eqnarray}
(\alpha+\beta+2)x-\beta+\alpha=(1-x^2)\frac{d^2}{dx^2}\mathcal{J}_N^{\alpha,\beta}(x)+\lambda_N^{\alpha,\beta} \mathcal{J}_N^{\alpha,\beta}(x).
\label{thm11}
\end{eqnarray}
By inputting (\ref{thm11}), we arrive at
\begin{eqnarray}
f'(x) &=& 2\frac{d}{dx}[\lambda_N^{\alpha,\beta} \mathcal{J}_N^{\alpha,\beta}(x) + (1-x^2)\frac{d^2}{dx^2}\mathcal{J}_N^{\alpha,\beta}(x)-x \cdot \frac{d}{dx}\mathcal{J}_N^{\alpha,\beta}(x)] \nonumber \\
&=& 2 [\frac{d}{dx} \mathcal{J}_N^{\alpha,\beta}(x)]^2 \cdot [(\alpha+\beta+1)x-\beta+\alpha].
\end{eqnarray}
Hence $x^*=\frac{\beta-\alpha}{\alpha+\beta+1}$ is the only point that may change the sign of $f'(x)$. Then
\begin{enumerate}
\item $|x^*| < 1 \ \Leftrightarrow \ (\alpha+\frac{1}{2})(\beta+\frac{1}{2})>0$,
\item $|x^*| \geqslant 1 \  \Leftrightarrow \ (\alpha+\frac{1}{2})(\beta+\frac{1}{2}) \leqslant 0$.
\end{enumerate}
Therefore, when $\alpha < 0$, $\beta \geqslant 0$, $f'(x)<0$ on $[-1,1]$, we have
\begin{eqnarray*}
\Vert \mathcal{J}_N^{\alpha,\beta}(x) \Vert_{L^{\infty}[-1,1]}=|\mathcal{J}_N^{\alpha,\beta}(-1)|=2^{|\alpha|}\binom{\tilde{n}+\beta}{\tilde{n}}.
\end{eqnarray*}
When $\alpha \geqslant 0$, $\beta < 0$, $f'(x)>0$ on $[-1,1]$, we have
\begin{eqnarray*}
\Vert \mathcal{J}_N^{\alpha,\beta}(x) \Vert_{L^{\infty}[-1,1]}=|\mathcal{J}_N^{\alpha,\beta}(1)|=2^{|\beta|}\binom{\tilde{n}+\alpha}{\tilde{n}}.
\end{eqnarray*}
When $\alpha=\beta<0$, if $\tilde{n}$ is even, then $x^*=0 \in A$ since $P^{-\alpha-1,-\alpha-1}_{\tilde{n}+1}(x)$ is an odd function, we have
\begin{eqnarray}
\Vert \mathcal{J}_N^{\alpha,\beta}(x) \Vert_{L^{\infty}[-1,1]}=|\mathcal{J}_N^{\alpha,\beta}(0)| &=& |P^{-\alpha,-\beta}_{\tilde{n}}(0)|=2^{-\tilde{n}} \binom{\tilde{n}-\alpha}{\tilde{n}/2}
\label{p0} \\
&=& 2^{-\tilde{n}} \frac{\Gamma(\tilde{n}-\alpha+1)}{\Gamma(\tilde{n}/2+1)\Gamma(\tilde{n}/2-\alpha+1)}, \nonumber
\end{eqnarray}
where (\ref{p0}) is derived from the three-recurrence formula of Jacobi polynomials. If $\tilde{n}$ is odd, since $x^*=0 \notin A$, and $\mathcal{J}_N^{\alpha,\beta}(x)$ is an odd function, the absolute maximal value is obtained at $\pm \eta_0$, the zero point closest to 0, and we have:
\begin{eqnarray*}
\Vert \mathcal{J}_N^{\alpha,\beta}(x) \Vert_{L^{\infty}[-1,1]}=|\mathcal{J}_N^{\alpha,\beta}(\eta_0)| \leqslant 2^{-\tilde{n}} \frac{\Gamma(\tilde{n}-\alpha+1)}{\Gamma(\tilde{n}/2+1)\Gamma(\tilde{n}/2-\alpha+1)}.
\end{eqnarray*}
\end{proof}

\subsection{Interpolation of Analytic Functions}
In this work, we always assume $u(x)$ is analytic on $[-1,1]$, and can be analytically extended to a certain domain on the complex-plane. The fundamental error analysis of Hermite interpolation was already provided in \cite{Interp+approx}.
\begin{lemma}[see \cite{Interp+approx}, Theorem 3.5.1]
Let $x_0, x_1, \cdots, x_n$ be $n+1$ distinct points in $[a,b]$. Let $m_0,m_1,\cdots, m_n$ be $n+1$ nonnegative integers. Let $N=(m_0+\cdots+m_n)+n$. Designate by $u_N$ the unique element of $\mathcal{P}_N$ for which
\begin{eqnarray}
u^{(k)}_N(x_i)=u^{(k)}(x_i), \ \ k=0,1,\cdots, m_i, \ \ i=0,1,\cdots,n,
\label{lemma2311}
\end{eqnarray}
where $u(x) \in C^N[a,b]$ and suppose that $u^{(N+1)}(x)$ exists in $(a,b)$. Then
\begin{eqnarray}
u(x)-u_N(x):=R_N(u;x)=\frac{u^{(N+1)}(\xi)}{(N+1)!}\omega_N(x),
\label{lemma2312}
\end{eqnarray}
where $\min(x,x_0,\cdots, x_m)<\xi<\max(x,x_0,\cdots,x_n)$, and
\begin{eqnarray}
\omega_N(x)=(x-x_0)^{m_0+1}(x-x_1)^{m_1+1}\cdots (x-x_n)^{m_n+1}.
\label{lemma2313}
\end{eqnarray}
\end{lemma}

\begin{lemma}[see \cite{Interp+approx}, Corollary 3.6.3]
Let $u(z)$ be analytic in a closed simply connected region $R$, $\mathcal{E}$ be a simple, closed, rectifiable curve that lies in $R$ and contains the distinct points $x_0, x_1, \cdots, x_n$ in tis interior. Then
\begin{eqnarray}
u_N(x)=\frac{1}{2\pi i} \oint_{\mathcal{E} } \frac{\omega_N(z)-\omega_N(x)}{\omega_N(z)(z-x)}u(z)dz
\end{eqnarray}
and
\begin{eqnarray}
R_N(u;x)=\frac{1}{2 \pi i} \oint_{\mathcal{E}} \frac{\omega_N(x)u(z)}{(z-x)\omega_N(z)}dz,
\label{lemma2332}
\end{eqnarray}
where $u_N(x)$, $R_N(u;x)$, $\omega_N(x)$ are defined in (\ref{lemma2311}), (\ref{lemma2312}) and (\ref{lemma2313}), respectively.
\end{lemma}

Without loss of generality, we usually consider functions that are analytic on the reference interval $[-1,1]$. It is known that each such function can be analytically extended to a domain enclosed by $Berstein \ ellipse$ $\mathcal{E}_{\rho}$ with the foci $\pm1$, namely,
\begin{eqnarray}
\mathcal{E}_{\rho}:=\{z: z=\frac{1}{2}(\rho e^{i\theta}+\rho^{-1} e^{-i \theta}), \ 0 \leqslant \theta < 2\pi \}, \ \rho > 1,
\label{bersteindef}
\end{eqnarray}
where $i=\sqrt{-1}$ is the imaginary unit, $\rho$ is the sum of semimajor and semiminor axes. Then we have the following bounds for $\mathcal{L}(\mathcal{E}_{\rho})$, the perimeter of the ellipse, and $\mathcal{D}_{\rho}$, the shortest distance from $\mathcal{E}_{\rho}$ to $[-1,1]$ respectively:
\begin{eqnarray}
\mathcal{L}(\mathcal{E}_{\rho}) \leqslant \pi (\rho+\rho^{-1})^{\frac{1}{2}}, \ \hbox{and} \ \mathcal{D}_{\rho}=\frac{1}{2}(\rho+\rho^{-1})-1.
\end{eqnarray}
For convenience, we define:
\begin{eqnarray}
M^{\rho}_u=\sup_{z \in \mathcal{E}_{\rho}} |u(z)|.
\end{eqnarray}

\section{Hermite Interpolation for Integer-order Derivatives}

In this work, we consider the three kinds of Hermite interpolations as follows.
\begin{enumerate}
\item Let $-1=x_0<x_1<\cdots<x_{\tilde{n}}<1$ be $\tilde{n}+1$ zeros of $\mathcal{J}^{0,-(k+1)}_{N+1}(x)$ (without considering multiplicity in all cases), where $k$ is a positive integer, $\tilde{n}=N-k$. We are going to find $u_{NL}(x) \in \mathcal{P}_N([-1,1])$, such that
\begin{eqnarray}
&& u^{(j)}_{NL}(x_0)=u^{(j)}(x_0), \ j=0,\cdots, k, \ \hbox{and} \nonumber\\
&& u_{NL}(x_i)=u(x_i), \ i=1,\cdots, \tilde{n}.
\label{def31}
\end{eqnarray}
\item Let $-1<x_0<x_1<\cdots<x_{\tilde{n}}=1$ be $\tilde{n}+1$ zeros of $\mathcal{J}^{-(k+1),0}_{N+1}(x)$, where $k$ is a positive integer, $\tilde{n}=N-k$. We are going to find $u_{NR}(x) \in \mathcal{P}_N([-1,1])$, such that
\begin{eqnarray}
&& u^{(j)}_{NR}(x_{\tilde{n}})=u^{(j)}(x_{\tilde{n}}), \ j=0,\cdots, k, \ \hbox{and} \nonumber\\
&& u_{NR}(x_i)=u(x_i), \ i=0,\cdots, \tilde{n}-1.
\label{def32}
\end{eqnarray}
\item Let $-1=x_0<x_1<\cdots<x_{\tilde{n}+1}=1$ be $\tilde{n}+2$ zeros of $\mathcal{J}^{-(k+1),-(k+1)}_{N+1}(x)$, where $k$ is a positive integer, $\tilde{n}=N-2k-1$. We are going to find $u_{NB}(x) \in \mathcal{P}_N([-1,1])$, such that
\begin{eqnarray}
&& u^{(j)}_{NB}(x_s)=u^{(j)}(x_s), \ j=0,1,\cdots, k, \ s=0,\tilde{n}+1,  \ \hbox{and} \nonumber\\
&& u_{NB}(x_i)=u(x_i), \ i=1,2,\cdots, \tilde{n}.
\label{def33}
\end{eqnarray}
\end{enumerate}
Clearly, if (\ref{def31}) is applied, then $x=-1$ will be a zero point of $(u-u_N)$ of multiplicity $k+1$; if (\ref{def32}) is applied, then $x=1$ will be a zero point of multiplicity $k+1$; if (\ref{def33}) is applied, then both $x=\pm1$ are zero points of multiplicity $k+1$.

Let us first discuss the interpolation errors. According to (\ref{lemma2312}) and Theorem \ref{thm1}, we  immediately derive the following corollary.

\begin{corollary}
Let $u_{NL}$ and $u_{NR}$ be the interpolants defined in (\ref{def31}) and (\ref{def32}) respectively, and $-(\min\{\alpha,0\}+\min\{\beta,0\}) \leqslant N$. Then
\begin{eqnarray}
|R_N(u;x)| \leqslant \frac{2^{N+1}\Gamma(N-k+1)}{\Gamma(2N-k+2)} \cdot \Vert u^{(N+1)}(x) \Vert_{C[-1,1]}.
\end{eqnarray}
Let $u_{NB}$ be the interpolant defined in (\ref{def33}). Then
\begin{eqnarray}
|R_N(u;x)| \leqslant  \frac{\Gamma(N-2k)\Gamma(N-k+1) \cdot \Vert u^{(N+1)}(x) \Vert_{C[-1,1]}}{\Gamma(2N-2k+1)\Gamma((N+1)/2-k)\Gamma((N+1)/2+1)}.
\end{eqnarray}
\end{corollary}

Besides corollary 3.1, another error representation is given by (\ref{lemma2332}). We now focus on finding the superconvergence points $\{ \xi_i \}$ in the sense of
\begin{eqnarray}
N^{\beta} \max_{-1\leqslant \xi_i \leqslant 1}|(u-u_{N})^{(k+1)}(\xi_i)| \lesssim \max_{-1\leqslant x \leqslant 1}|(u-u_{N})^{(k+1)}(x)|,
\label{defspcvg}
\end{eqnarray}
where $u_{N}$ is one of $u_{NL}, u_{NR}, u_{NB}$, $k$ is the number given by (\ref{def31}), (\ref{def32}), (\ref{def33}) respectively, and $\beta$ is the gain of convergence rate. The superconvergence points $\{ \xi_i \}$ and $\beta$ are determined by the type of interpolants. Following the analysis in \cite{Zhang+2012+SINUM}, we can obtain the following theorem.


\begin{theorem}
Let $u(x)$ be analytic on [-1,1] and within $Berstein \ Ellipse$ $\mathcal{E}_{\rho}$ defined in (\ref{bersteindef}) with $\rho>1$, and let $u_{NL}$, $u_{NR}$, $u_{NB}$ be the interpolants defined in (\ref{def31}), (\ref{def32}) and (\ref{def33}) respectively. Suppose $k << N$, then we obtain the global error estimates as follows. For the interpolant $u_{NL}$, we have
\begin{eqnarray}
\max_{-1\leqslant x \leqslant 1}|(u-u_{NL})^{(k+1)}(x)| \leqslant c_1 M^{\rho}_u   (1+\frac{D_{\rho}}{2} N^2)^{k+1}\frac{\tilde{n}^{\frac{1}{2}}}{\rho^{\tilde{n}}},
\end{eqnarray}
the superconvergence points $\{{\xi_L^i} \}_{i=1}^{N-k}$ are zero points of $P^{k+1,0}_{N-k}(x)$, and
\begin{eqnarray}
\max_{1\leqslant i \leqslant N-k}|(u-u_{NL})^{(k+1)}(\xi_L^i)| \leqslant c_1^{'} M^{\rho}_u   (1+\frac{D_{\rho}}{2} N^2)^{k}\frac{\tilde{n}^{\frac{1}{2}}}{\rho^{\tilde{n}}}.
\end{eqnarray}
For the interpolant $u_{NR}$, we have
\begin{eqnarray}
\max_{-1\leqslant x \leqslant 1}|(u-u_{NR})^{(k+1)}(x)| \leqslant c_1 M^{\rho}_u   (1+\frac{D_{\rho}}{2} N^2)^{k+1}\frac{\tilde{n}^{\frac{1}{2}}}{\rho^{\tilde{n}}},
\end{eqnarray}
the superconvergence points $\{\xi_R^i \}_{i=1}^{N-k}$ are zero points of $P^{0,k+1}_{N-k}(x)$, and
\begin{eqnarray}
\max_{1\leqslant i \leqslant N-k}|(u-u_{NR})^{(k+1)}(\xi_R^i)| \leqslant c_1^{'} M^{\rho}_u   (1+\frac{D_{\rho}}{2} N^2)^{k} \frac{\tilde{n}^{\frac{1}{2}}}{\rho^{\tilde{n}}}.
\end{eqnarray}
For the interpolant $u_{NB}$, we have
\begin{eqnarray}
\max_{-1\leqslant x \leqslant 1}|(u-u_{NB})^{(k+1)}(x)| \leqslant c_2 M^{\rho}_u (N^{k+1}+O(N^{k-\frac{1}{2}}) ) \frac{\tilde{n}^{\frac{1}{2}}}{\rho^{\tilde{n}}},
\label{thm15}
\end{eqnarray}
the superconvergence points $\{\xi_B^i \}_{i=1}^{N-k}$ are zero points of $P^{0,0}_{N-k}(x)$, and
\begin{eqnarray}
\max_{1\leqslant i \leqslant N-k}|(u-u_{NB})^{(k+1)}(\xi_B^i)| \leqslant c_2^{'} M^{\rho}_u  (1+\frac{2}{e}N)^k N^{-\frac{1}{2}}\frac{\tilde{n}^{\frac{1}{2}}}{\rho^{\tilde{n}}},
\label{thm16}
\end{eqnarray}
where $c_i$, $c_i^{'}$, $i=1,2$, only depend on $\rho$, $k$.
\label{thm2}
\end{theorem}

\begin{proof}
Since $u_{NL}$ and $u_{NR}$ are completely symmetric, without loss of generality, we only consider the cases of $u_{NL}$ and $u_{NB}$ in the proof. According to (\ref{lemma2332}), we have
\begin{eqnarray}
&&D^{k+1} (u(x)-u_N(x))=D^{k} R_N(u;x) \label{differror} \\
&=&\frac{1}{2\pi i} \oint_{\mathcal{E}_\rho} D^{k+1} \left( \frac{\omega_{N+1}(x)}{z-x} \right)\frac{u(z)}{\omega_{N+1}(z)} dz \nonumber \\
&=&\frac{1}{2 \pi i} \sum_{l=0}^{k+1} \binom{k+1}{l} \oint_{\mathcal{E}_\rho} \frac{D^{l} \omega_{N+1}(x)} {(z-x) ^{k-l+2}}\frac{u(z)}{\omega_{N+1}(z)} dz,\quad \quad \forall x \in[-1,1].
\label{error}
\end{eqnarray}
If $u_N(x)=u_{NL}(x)$, then $\omega_{N+1}(x)=(A_{\tilde{n}}^{0,k+1})^{-1}\mathcal{J}^{0,-(k+1)}_{N+1}(x)$, where $A_{\tilde{n}}^{0,k+1}$ denotes the leading coefficient. Since there is a $\omega_{N+1}(z)$ in the denominator, we just consider $\omega_{N+1}(x)=\mathcal{J}^{0,-(k+1)}_{N+1}(x)$. According to Lemma \ref{lemma1}, we have
\begin{eqnarray}
D^{l} \omega_{N+1}(x)=\frac{\Gamma(N+2)}{\Gamma(N-l+2)} \mathcal{J}^{l,-k-1+l}_{N+1-l}(x), \ l=0,1,\cdots k+1.
\end{eqnarray}
When $N$ is large enough ($N \geqslant 3$), we arrive at
\begin{eqnarray}
\max_{-1\leqslant x\leqslant1} |D^{l}\omega_{N+1}(x)| =  \frac{\Gamma(N+2)}{\Gamma(N-l+2)} 2^{k+1-l}\binom{\tilde{n}+l}{\tilde{n}} \leqslant 2^{k+2} (\frac{N^2}{2})^l,
\end{eqnarray}
where $\tilde{n}=N-k$. On the other hand, according to the analysis in (\cite{LWang+2014+JSC}, (4.7)), $\forall \rho>1, k \in \mathbb{N}$, $\exists C_1(\rho,k)>0$, such that
\begin{eqnarray}
\min_{z \in \mathcal{E}_{\rho}} |P^{0,k+1}_{\tilde{n}}(z)| \geqslant C_1(\rho,k) \tilde{n}^{-\frac{1}{2}} \rho^{\tilde{n}}.
\end{eqnarray}
Noting that
$$\min_{z \in \mathcal{E}_{\rho}} |(1+z)^{k+1}| = D_{\rho}^{k+1}$$
and from (\ref{def2.3}), we have
\begin{eqnarray}
\min_{z \in \mathcal{E}_{\rho}} |\mathcal{J}^{0,-(k+1)}_{N+1} (z)| &\geqslant& \min_{z \in \mathcal{E}_{\rho}} |(1+z)^{k+1}| \cdot \min_{z \in \mathcal{E}_{\rho}} |P^{0,k+1}_{\tilde{n}}(z)| \nonumber\\
& \geqslant&  C_1(\rho,k) D^{k+1}_{\rho} \tilde{n}^{-\frac{1}{2}} \rho^{\tilde{n}}.
\end{eqnarray}
The identity (\ref{error}) thus turns out to be
\begin{eqnarray}
&=& \frac{1}{2 \pi i} \sum_{l=0}^{k+1} \binom{k+1}{l} \frac{\Gamma(N+2)}{\Gamma(N-l+2)} \oint_{\mathcal{E}_\rho} \frac{\mathcal{J}^{l,-k-1+l}_{N+1-l}(x)} {(z-x) ^{k+2-l}}\frac{u(z)}{\mathcal{J}^{0,-(k+1)}_{N+1}(z)} dz \label{thm1p1}\\
&\leqslant& \frac{1}{2 \pi} \sum_{l=0}^{k+1} \binom{k+1}{l} \frac{\Gamma(N+2)}{\Gamma(N-l+2)} \oint_{\mathcal{E}_\rho} \frac{|\mathcal{J}^{l,-k-1+l}_{N+1-l}(x)|} {|(z-x)| ^{k+2-l}}\frac{|u(z)|}{|\mathcal{J}^{0,-(k+1)}_{N+1}(z)|} d|z| \nonumber \\
&\leqslant&  \frac{1}{2 \pi} \sum_{l=0}^{k+1} \binom{k+1}{l} 2^{k+2} (\frac{N^2}{2})^l \frac{M^{\rho}_u \mathcal{L}(\mathcal{E}_{\rho})}{D_{\rho}^{k+2-l}} \cdot \frac{\tilde{n}^{\frac{1}{2}}}{C_1(\rho,k) D_{\rho}^{k+1}  \rho^{\tilde{n}}} \nonumber\\
&=&M^{\rho}_u\frac{2^{k+1} \mathcal{L}(\mathcal{E}_{\rho})}{\pi C_1(\rho,k) D_{\rho}^{2k+3} }  \frac{\tilde{n}^{\frac{1}{2}}}{\rho^{\tilde{n}}} \sum_{l=0}^{k+1} \binom{k+1}{l} (\frac{D_{\rho}}{2}N^2)^l \nonumber\\
&=& M^{\rho}_u\frac{2^{k+1} \mathcal{L}(\mathcal{E}_{\rho})}{\pi C_1(\rho,k) D_{\rho}^{2k+3}}  \frac{\tilde{n}^{\frac{1}{2}}}{\rho^{\tilde{n}}} (1+\frac{D_{\rho}}{2} N^2)^{k+1}.
\end{eqnarray}
When $x=\xi_L^i$, $i=1,\cdots,N-k$, where $\{\xi_L^i \}_{i=1}^{N-k}$ are zero points of $P^{k+1,0}_{N-k}(x)$, the last term in (\ref{thm1p1}) vanishes. This is, for $1\leqslant i \leqslant N-k$, we have
\begin{eqnarray}
&&|D^{k+1} (u(\xi_L^i)-u_N(\xi_L^i))| \nonumber\\
&=& \frac{1}{2 \pi i} \sum_{l=0}^{k} \binom{k+1}{l} \frac{\Gamma(N+2)}{\Gamma(N-l+2)} | \oint_{\mathcal{E}_\rho} \frac{\mathcal{J}^{l,-k-1+l}_{N+1-l}(\xi_L^i)} {(z-\xi_L^i) ^{k+2-l}}\frac{u(z)}{\mathcal{J}^{0,-(k+1)}_{N+1}(z)} dz|  \nonumber \\
&\leqslant& M^{\rho}_u\frac{2^{k+1} \mathcal{L}(\mathcal{E}_{\rho})}{\pi C_1(\rho,k) D_{\rho}^{2k+3}}  \frac{\tilde{n}^{\frac{1}{2}}}{\rho^{\tilde{n}}} \sum_{l=0}^{k} \binom{k+1}{l} (\frac{D_{\rho}}{2}N^2)^l \\
&\leqslant& M^{\rho}_u\frac{2^{k+1} (k+1) \mathcal{L}(\mathcal{E}_{\rho})}{\pi C_1(\rho,k) D_{\rho}^{2k+3}}  \frac{\tilde{n}^{\frac{1}{2}}}{\rho^{\tilde{n}}}  (1+\frac{D_{\rho}}{2} N^2)^{k}.
\end{eqnarray}
On the other hand, if $u_N(x)=u_{NB}(x)$, then $\omega_{N+1}(x)=(A_{\tilde{n}}^{k+1,k+1})^{-1}\mathcal{J}^{-(k+1),-(k+1)}_{N+1}(x)$. Similarly, we have:
\begin{eqnarray}
D^{l} \omega_{N+1}(x)=(-2)^l \frac{\Gamma(N-2k+l)}{\Gamma(N-2k)} \mathcal{J}^{l-k-1,l-k-1}_{N+1-l}(x), \ l=0,1,\cdots k+1.
\end{eqnarray}
when $N$ is large enough. By Stirling's formula, we achieve
\begin{eqnarray}
\max_{-1\leqslant x\leqslant1} |D^{l}\omega_{N+1}(x)| &\leqslant&  \frac{\Gamma(N-2k+l) \Gamma(N-k+1+l)}{2^{N-2k-1}\Gamma(N-2k) \Gamma(\frac{N-2k+1+l}{2}) \Gamma(\frac{N+3+l}{2})}
\label{polyesx20} \\
&\leqslant& \frac{2^{k+2}}{\sqrt{\pi}} (\frac{2}{e})^l N^{l-\frac{1}{2}} , \ l=0,\cdots, k. \label{polyesx21}
\end{eqnarray}
Taking $l=k+1$ yields
\begin{eqnarray}
\max_{-1\leqslant x\leqslant1} |D^{k+1}\omega_{N+1}(x)| &=& 2^{k+1} \frac{\Gamma(N-k+1)}{\Gamma(N-2k)} \max_{-1\leqslant x\leqslant 1}| P^{0,0}_{N-k}(x)| \nonumber\\
&\leqslant& 2^{k+1} N^{k+1}. \label{polyesx22}
\end{eqnarray}
Again, $\forall \rho>1, k \in \mathbb{N}$, $\exists C_2(\rho,k)>0$, such that
\begin{eqnarray*}
\min_{z \in \mathcal{E}_{\rho}} |P^{k+1,k+1}_{\tilde{n}}(z)| \geqslant C_2(\rho,k) \tilde{n}^{-\frac{1}{2}} \rho^{\tilde{n}},
\end{eqnarray*}
and for $\forall z \in \mathcal{E}_{\rho}$, we have
\begin{eqnarray*}
|(1+z)(1-z) | &=& \frac{1}{4} |(\rho e^{i \theta}-\rho^{-1}e^{-i\theta})^2|=\frac{1}{4} |(\rho-\rho^{-1})\cos \theta+ i (\rho+\rho^{-1})\sin \theta |^2 \\
&=& \frac{1}{4} [ (\rho^2+\rho^{-2}-2) \cos^2 \theta +(\rho^2+\rho^{-2}+2) \sin^2 \theta ] \\
&=& \frac{1}{4} (\rho-\rho^{-1})^2 +  \sin^2 \theta \geqslant \frac{1}{4}(\rho-\rho^{-1})^2.
\end{eqnarray*}
Therefore
\begin{eqnarray}
\min_{z \in \mathcal{E}_{\rho}} |\mathcal{J}^{-(k+1),-(k+1)}_{N+1} (z)| &\geqslant& \min_{z \in \mathcal{E}_{\rho}} |(1+z)(1-z)|^{k+1} \cdot \min_{z \in \mathcal{E}_{\rho}} |P^{k+1,k+1}_{\tilde{n}}(z)| \nonumber \\
& \geqslant&  \frac{(\rho-\rho^{-1})^{2k+2}}{4^{k+1}} C_2(\rho,k)  \tilde{n}^{-\frac{1}{2}} \rho^{\tilde{n}}.
\label{polyesz2}
\end{eqnarray}
Similar to previous case, by inputing (\ref{polyesx21}), (\ref{polyesx22}) and (\ref{polyesz2}) into (\ref{error}), the proof is complete.
\end{proof}


\begin{remark}
Let $f(N;k,l)$ be defined by the right hand side in \eqref{polyesx20}, i.e.
\begin{eqnarray}
f(N;k,l) := \frac{\Gamma(N-2k+l) \Gamma(N-k+1+l)}{2^{N-2k-1}\Gamma(N-2k) \Gamma(\frac{N-2k+1+l}{2}) \Gamma(\frac{N+3+l}{2})}.
\end{eqnarray}
\begin{figure}[htbp] 
   \centering
    \includegraphics[width=2.5in]{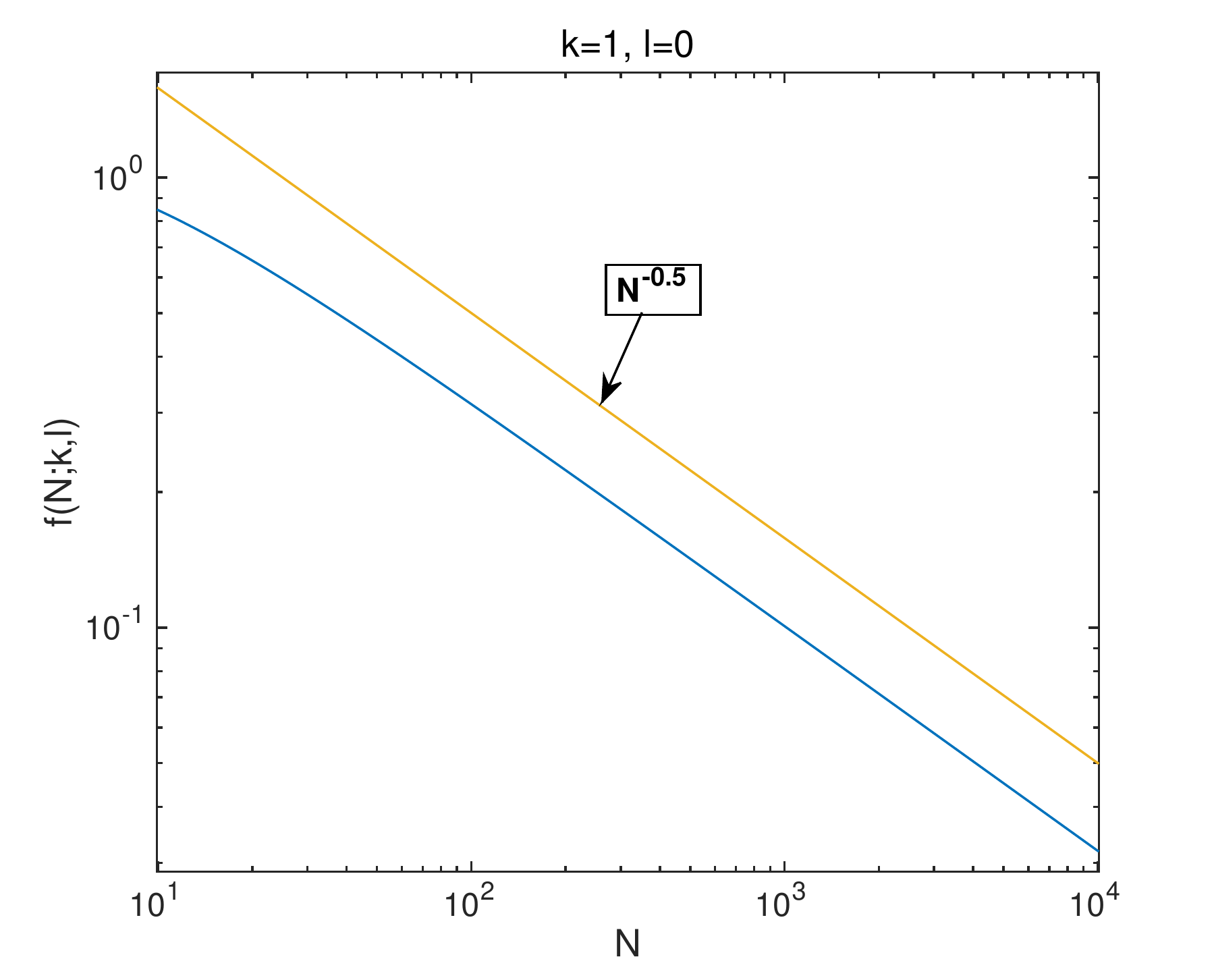}
    \includegraphics[width=2.5in]{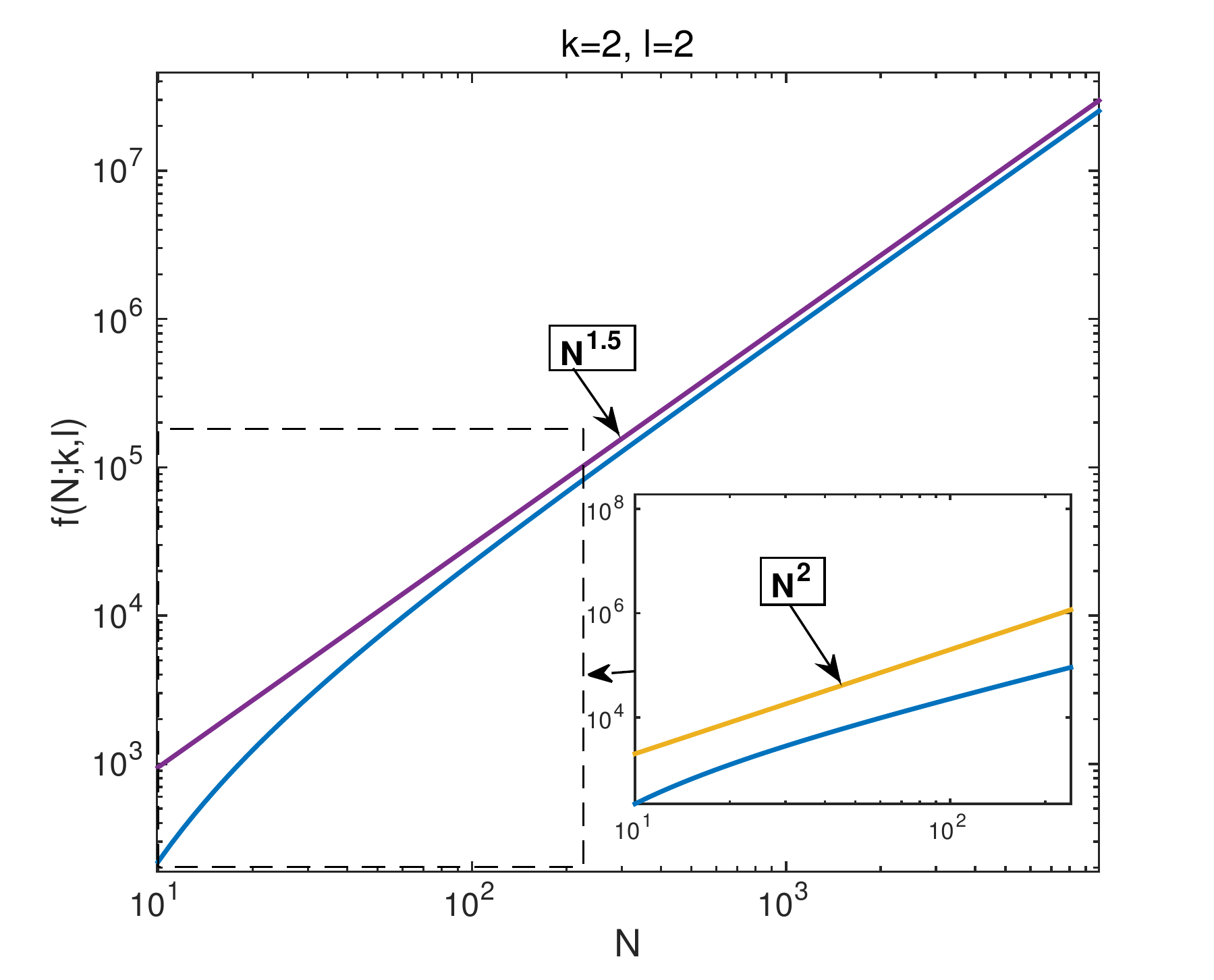}
   \caption{The graphs of $f(N;k,l)$ for $N \in [10,10^4]$, where $k=1$, $l=0$ (Left), $k=2$, $l=2$ (Right).}
   \label{fg31}
\end{figure}
In Fig. \ref{fg31}, it shows that $f(N;k,l)=O(N^{l-\frac{1}{2}})$ when $N$ is large enough. However, when $N$ is small, especially $N \leqslant 30$, it is more likely that $$f(N;k,l) \lesssim O(N^{l}).$$
Therefore, even though \eqref{thm15} and \eqref{thm16} in Theorem \ref{thm2} show that the gain of convergence rate at superconvergence points is $O(N^{-\frac{3}{2}})$ when $N$ is large enough. While $N \leqslant 30$, it seems the gain of convergence rate behaves like $O(N^{-1})$.
\end{remark}

\begin{remark}
Similar to Chebyshev and Legendre interpolation, the three kinds of Hermite interpolations are exponentially convergent as well. Strictly speaking, they are exponentially convergent with respect to $\tilde{n}$, the number of interpolation points except for $x=\pm1$, instead of $N$, the degree of interpolants. Therefore, it will significantly improve the accuracy to increase the number of inside interpolation points. On the other hand, for any fixed $k$, because of the exponential convergence rate, $u_N$ converges to $u$ very fast. So admissible results can be usually obtained when $N$ is small .
\end{remark}

Since all of the three kinds of superconvergence points are zero points of Jacobi polynomials, they can be efficiently calculated.

\section{Hermite Interpolation for Riemann Liouville Derivatives with Arbitrary Positive Order}

\subsection{Theoretical Statements}
As mentioned in introduction, the Hermite interpolations are mainly designed to resolve the singularities of the fractional differential operators. Let $k<\alpha <k+1$, where $k$ is a nonnegative integer. For the left Riemann Liouville fractional operator $_{-1}D_x^{\alpha}$ which is defined in (\ref{def11}), one can obviously observe that
\begin{eqnarray}
(u-u_{NL})^{(l)}(-1)=0, \ l=0,1,\cdots, k,
\label{chap41}
\end{eqnarray}
and $_{-1}D_x^{\alpha}(u(x)-u_{NL}(x))$ is bounded on $[-1,1]$. As for the right Riemann-Liouville fractional operator $_{x}D_1^{\alpha}$ defined in (\ref{def12}), similarly, we have
\begin{eqnarray}
(u-u_{NR})^{(l)}(1)=0, \ l=0,1,\cdots, k,
\end{eqnarray}
and $_{x}D_1^{\alpha}(u(x)-u_{NR}(x))$ is bounded on $[-1,1]$ as well.
\begin{remark}
If interpolating in these ways, it is equivalent to take Riemann-Liouville derivative and Caputo derivative, so the following results also work for the Caputo fractional derivative.
\end{remark}

\begin{theorem}
Let $k<\alpha <k+1$, where $k$ is a nonnegative integer. Let $u_{NL}$ and $u_{NR}$ be the interpolants defined in (\ref{def31}) and (\ref{def32}) respectively. Then
\begin{eqnarray}
\Vert {_{-1}D_x^{\alpha}}(u-u_{NL}) \Vert_{L^{\infty}[-1,1]} \leqslant \frac{2^{k+1-\alpha}}{\Gamma(k+2-\alpha)} \Vert (u-u_{NL})^{(k+1)}\Vert_{L^{\infty}[-1,1]}; \label{thm21}\\
\Vert {_{x}D_1^{\alpha}}(u-u_{NR}) \Vert_{L^{\infty}[-1,1]} \leqslant \frac{2^{k+1-\alpha}}{\Gamma(k+2-\alpha)} \Vert (u-u_{NR})^{(k+1)}\Vert_{L^{\infty}[-1,1]}. \label{thm22}
\end{eqnarray}
\label{thm3}
\end{theorem}

\begin{proof}
Without loss of generality, we only consider the poof of left Riemann Liouville derivatives. According to (\ref{def11}), (\ref{chap41}) and Holder's Inequality yields
\begin{eqnarray}
&&|_{-1}D_x^{\alpha}(u(x)-u_{NL}(x))| \nonumber \\
&=& \frac{1}{\Gamma(k+1-\alpha)} \Big \vert \int_{-1}^{x}\frac{(u-u_{NL})^{(k+1)}(t)}{(x-t)^{\alpha-k}}dt \Big| \nonumber \\
&\leqslant& \frac{\Vert (u-u_{NL})^{(k+1)} \Vert_{L^{\infty}}}{\Gamma(k+1-\alpha)} \int_{-1}^x (x-t)^{k-\alpha}dt \nonumber \\
&=&\frac{(1+x)^{k+1-\alpha}}{\Gamma(k+2-\alpha)} \Vert (u-u_{NL})^{(k+1)} \Vert_{L^{\infty}[-1,1]}, \forall x \in [-1,1].\label{thm2p}
\end{eqnarray}
Therefore,
\begin{eqnarray*}
\Vert {_{-1}D_x^{\alpha}}(u-u_{NL}) \Vert_{L^{\infty}[-1,1]} \leqslant \frac{2^{k+1-\alpha}}{\Gamma(k+2-\alpha)} \Vert (u-u_{NL})^{(k+1)}\Vert_{L^{\infty}[-1,1]}.
\end{eqnarray*}
\end{proof}
\begin{remark}
Let $f(x)=\frac{2^{x}}{\Gamma(x+1)}$. Then $f(0)=1$, $f(1)=2$, so $f(x)$ is uniformly continuous and bounded on $[1,2]$. Then we have
\begin{eqnarray}
\lim_{\alpha \to k} \frac{2^{k+1-\alpha}}{\Gamma(k+2-\alpha)}=1, \ \text{and} \ \lim_{\alpha \to k+1} \frac{2^{k+1-\alpha}}{\Gamma(k+2-\alpha)}=2.
\end{eqnarray}
Therefore, $\exists C>0$, such that $\forall \alpha \in (k,k+1)$, we always have: $$\frac{2^{k+1-\alpha}}{\Gamma(k+2-\alpha)} \leqslant C=2.$$
\end{remark}
Parallel to the conclusion in \cite{Zhao+2016+SISC}, we have the following theorem.

\begin{theorem}
Let $\alpha \in (k,k+1)$, let $u_{NL}$ and $u_{NR}$ be the interpolants defined in (\ref{def31}) and (\ref{def32}) respectively. The $\alpha$-th left Riemann Liouville fractional derivative superconverges at $\{ _L\xi_i^{\alpha} \}$ satisfying
\begin{eqnarray}
_{-1} I_x^{k+1-\alpha} P^{k+1,0}_{N-k}(_L\xi_i^{\alpha})=0,\ i=1,\cdots,N-k.
\label{superconrllz}
\end{eqnarray}
Similarly, the $\alpha$-th right Riemann Liouville fractional derivative superconverges at $\{ _R\xi_i^{\alpha} \}$ satisfying
\begin{eqnarray}
_x I_1^{k+1-\alpha} P^{0,k+1}_{N-k}(_R\xi_i^{\alpha})=0,\ i=1,\cdots,N-k.
\label{superconrlrz}
\end{eqnarray}
\label{thm6}
\end{theorem}

\begin{proof} Again, we only give the proof of (\ref{superconrllz}). The proof starts with (\ref{lemma2332}). Similar to (\ref{differror}) for $\forall x \in[-1,1]$, we have
\begin{eqnarray}
&&_{-1}D_x^{\alpha} (u(x)-u_N(x)) \nonumber \\
&=&\frac{1}{2\pi i} \oint_{\mathcal{E}_\rho} {_{-1}D_x^{\alpha}} \Big( \frac{\omega_{N+1}(x)}{z-x} \Big)\frac{u(z)}{\omega_{N+1}(z)} dz \nonumber \\
&=&\frac{1}{2 \pi i} \oint_{\mathcal{E}_\rho}  \sum_{m=0}^{\infty} \frac{\Gamma(\alpha+1)}{\Gamma(\alpha-m+1)} \frac{_{-1} D_x^{\alpha-m} \mathcal{J}^{0,-(k+1)}_{N+1}(x)} {(z-x) ^{(m+1)}}\frac{u(z)}{\mathcal{J}^{0,-(k+1)}_{N+1}(z)} dz.
\end{eqnarray}
According to the analysis in \cite{Zhao+2016+SISC}, the decay of the error is dominated by the leading term. By (\ref{def11}), (\ref{lemma242}), and (\ref{lemma243}), we arrive at
\begin{eqnarray}
&& \oint_{\mathcal{E}_\rho} \frac{_{-1} D_x^{\alpha} \mathcal{J}^{0,-(k+1)}_{N+1}(x)} {(z-x) }\frac{u(z)}{\mathcal{J}^{0,-(k+1)}_{N+1}(z)} dz \\
&=& \oint_{\mathcal{E}_\rho} \frac{_{-1} I_x^{k+1-\alpha} P^{k+1,0}_{N-k}(x)} {(z-x) }\frac{u(z)}{\mathcal{J}^{0,-(k+1)}_{N+1}(z)} dz.
\end{eqnarray}
When $x={_L\xi_i^{\alpha}}, i=1,\cdots,N-k$, the leading term vanishes, and the remaining terms have higher-order convergence rates.
\end{proof}
\\

\subsection{Calculating the Superconvergence Points}
The recurrence formulas of calculating $_{-1} I_x^{\gamma} P^{a,b}_{n}(x)$ and $_{x} I_1^{\gamma} P^{a,b}_{n}(x)$ are provided in (20) and (23) respectively in \cite{Zeng+Li+arxiv}. Therefore, we can calculate the zero points of $_{-1} I_x^{k+1-\alpha} P^{k+1,0}_{N-k}(x)$ and $_x I_1^{k+1-\alpha} P^{0,k+1}_{N-k}(x)$ by eigenvalue method.

\begin{theorem}
The zeros $\{_L\xi_i \}^n_{i=1}$ of the function $_{-1} I_x^{\gamma} P^{a,b}_{n}(x)$ are eigenvalues of the following matrix $S_L \in \mathcal{M}_{n\times n}(\mathbb{R})$:
\begin{equation}
\left(
\begin{array}{cccccc}
S_{L11} & S_{L12} &  & \cdots &  & 0 \\
S_{L21} & S_{L22} & S_{L23} &  &  & \\
S_{L31} & S_{L32} & S_{L33} & S_{L34} & & \vdots  \\
S_{L41} & 0 & S_{L43} & S_{L44} & \ddots &  \\
 \vdots &  \vdots & 0 & \ddots & \ddots & S_{Ln-1,n} \\
S_{Ln1} & 0 & \cdots & 0 & S_{Ln,n-1} &S_{Lnn}
\end{array} \right),
\end{equation}
where
\begin{eqnarray*}
&& S_{L11}=\frac{b-a+2\gamma+2b\gamma}{a+b+2};   \ \ S_{Lkk}=\frac{A^2_k}{A^1_k}, k=2,\cdots, n;  \\
&& S_{L12}=\frac{2(1+\gamma)}{a+b+2};   \ \ \ S_{Lk,k+1}=\frac{1}{A^1_k},  k=2,\cdots,n-1;   \\
&& S_{L21}=\frac{A^3_2-A^{L4}_2}{A^1_2}; \ \ S_{Lk+1,k}=\frac{A^3_{k+1}}{A^1_{k+1}}, k=2,\cdots, n-1; \\
&& S_{Lk,1}=-\frac{A^{L4}_k}{A^1_k}, k=3,\cdots, n.
\end{eqnarray*}
The zeros $\{_R\xi_i \}^n_{i=1}$ of the function $_x I_1^{\gamma} P^{a,b}_{n}(x)$ are eigenvalues of the matrix $S_R \in \mathcal{M}_{n\times n}(\mathbb{R})$, which is of the same form as $S_L$, where
\begin{eqnarray*}
&& S_{R11}=\frac{b-a-2\gamma-2a\gamma}{a+b+2};   \ \ S_{Rkk}=\frac{A^2_k}{A^1_k}, k=2,\cdots, n;  \\
&& S_{R12}=\frac{2(1+\gamma)}{a+b+2};   \ \ \ S_{Rk,k+1}=\frac{1}{A^1_k},  k=2,\cdots,n-1;   \\
&& S_{R21}=\frac{A^3_2-A^{R4}_2}{A^1_2}; \ \ S_{Rk+1,k}=\frac{A^3_{k+1}}{A^1_{k+1}}, k=2,\cdots, n-1; \\
&& S_{Rk,1}=-\frac{A^{R4}_k}{A^1_k}, k=3,\cdots, n,
\end{eqnarray*}
when $2 \leqslant k \leqslant n$,
\begin{eqnarray*}
&& A^1_k=\frac{\tilde{a}_k}{1+\gamma \tilde{a}_k \bar{c}_k}, \
A^2_k=\frac{\tilde{b}_k+\gamma \tilde{a}_k\bar{b}_k}{1+\gamma \tilde{a}_k \bar{c}_k}, \ A^3_k=\frac{\tilde{c}_k+\gamma \tilde{a}_k \bar{a}_k}{1+\gamma \tilde{a}_k \bar{c}_k}, \\
&& A^{L4}_k=\frac{\gamma \tilde{a}_k(\bar{a}_k d^{L1}_k + \bar{b}_k d^{L2}_k + \bar{c}_k d^{L3}_k)}{(1+\gamma \tilde{a}_k \bar{c}_k) }, \ A^{R4}_k=\frac{\gamma \tilde{a}_k(\bar{a}_k d^{R1}_k + \bar{b}_k d^{R2}_k + \bar{c}_k d^{R3}_k)}{(1+\gamma \tilde{a}_k \bar{c}_k) },
\end{eqnarray*}
where
\begin{eqnarray*}
&& \tilde{a}_k=\frac{(2k-1+a+b)(2k+a+b)}{2k(k+a+b)}, \tilde{b}_k=\frac{(b^2-a^2)(2k-1+a+b)}{2k(k+a+b)(2k-2+a+b)},\\
&& \tilde{c}_k=\frac{(k-1+a)(k-1+b)(2k+a+b)}{k(k+a+b)(2k-2+a+b)},\\
&& \bar{a}_k=\frac{-2(k-1+a)(k-1+b)}{(k-1+a+b)(2k-2+a+b)(2k-1+a+b)},\\
&& \bar{b}_k=\frac{2(a-b)}{(2k-2+a+b)(2k+a+b)}, \bar{c}_k=\frac{2(k+a+b)}{(2k-1+a+b)(2k+a+b)},\\
&& d^{L1}_k=(-1)^{k}\binom{k-2+b}{k-2}, \ d^{L2}_k=(-1)^{k-1}\binom{k-1+b}{k-1}, \ d^{L3}_k=(-1)^k \binom{k+b}{k},\\
&& d^{R1}_k=\binom{k-2+a}{k-2}, \ d^{R2}_k=\binom{k-1+a}{k-1}, \ d^{R3}_k= \binom{k+a}{k}.\\
\end{eqnarray*}

\label{thm4}
\end{theorem}
\begin{proof}
The recurrence formula for left integral of Jacobi polynomial, $_{-1} I_x^{\gamma} P^{a,b}_{n}(x)$ is given by, seeing \cite{Zeng+Li+arxiv} (20):
\begin{eqnarray}
\left\{\begin{array}{lll}
{_{-1} I_x^{\gamma} P^{a,b}_{0}(x)}&=&\frac{(1+x)^{\gamma}}{\Gamma(\gamma+1)}, \\
{_{-1} I_x^{\gamma} P^{a,b}_{1}(x)}&=&\frac{a+b+2}{2}(x [ {_{-1} I_x^{\gamma} P^{a,b}_{0}(x)}]-\frac{\gamma (1+x)^{\gamma+1}}{\Gamma(\gamma+2)})\\
&&+\frac{a-b}{2} [  {_{-1} I_x^{\gamma} P^{a,b}_{0}(x)}],\\
{_{-1} I_x^{\gamma} P^{a,b}_{k+1}(x)} &=& (A^1_{k+1} x-A^2_{k+1}) [ {_{-1} I_x^{\gamma} P^{a,b}_{k}(x)}] -A^3_{k+1} [{_{-1} I_x^{\gamma} P^{a,b}_{k-1}(x)}] \\
&&+A^{L4}_{k+1} [ {_{-1} I_x^{\gamma} P^{a,b}_{0}(x)}], \ \ \ k=1,\cdots, n-1,
\end{array}\right.
\end{eqnarray}
and the formula for right integral, $_x I_1^{\gamma} P^{a,b}_{n}(x)$, is given by, seeing \cite{Zeng+Li+arxiv} (23):
\begin{eqnarray}
\left\{\begin{array}{lll}
{_x I_1^{\gamma} P^{a,b}_{0}(x)}&=&\frac{(1-x)^{\gamma}}{\Gamma(\gamma+1)}, \\
{_x I_1^{\gamma} P^{a,b}_{1}(x)}&=&\frac{a+b+2}{2}(x [ {_x I_1^{\gamma} P^{a,b}_{0}(x)}]+\frac{\gamma (1-x)^{\gamma+1}}{\Gamma(\gamma+2)})\\
&&+\frac{a-b}{2} [  {_x I_1^{\gamma} P^{a,b}_{0}(x)}],\\
{_{-1} I_x^{\gamma} P^{a,b}_{k+1}(x)} &=& (A^1_{k+1} x-A^2_{k+1}) [ {_x I_1^{\gamma} P^{a,b}_{k}(x)}] -A^3_{k+1} [{_x I_1^{\gamma} P^{a,b}_{k-1}(x)}] \\
&&+A^{R4}_{k+1} [ {_x I_1^{\gamma} P^{a,b}_{0}(x)}], \ \ \ k=1,\cdots, n-1,
\end{array}\right.
\end{eqnarray}
where $A^1_k, A^2_k, A^3_k, A^{L4}_k, A^{R4}_k$, $k=1,2,\cdots, n-1$, are given in the statement of the theorem.  The rest of proof is parallel to the Theorem 3.4 in \cite{shen+wang+tang+spectral}.
\end{proof}

\section{Hermite Interpolation for Riesz Derivatives with Arbitrary Positive Order}
Parallel to the previous section, there are similar conclusions for Riesz fractional derivatives.

\begin{theorem}
Let $k<\alpha <k+1$, where $k$ is a nonnegative integer. Let $u_{NB}$ be the interpolant defined in (\ref{def33}) respectively. Then we have:
\begin{eqnarray}
\Vert {^RD^{\alpha}}(u-u_{NL}) \Vert_{L^{\infty}[-1,1]} \leqslant \frac{2C_R}{\Gamma(k+2-\alpha)} \Vert (u-u_{NL})^{(k+1)}\Vert_{L^{\infty}[-1,1]}, \label{thm41}
\end{eqnarray}
where $C_R=\max\{c1,c2\}$, both $c1,c2$ are defined in (\ref{def22}).
\label{thm5}
\end{theorem}
\begin{proof}
It is a direct corollary of (\ref{def22}), (\ref{thm2p}) and the fact that:
\begin{eqnarray*}
\max_{-1\leqslant x \leqslant 1} \{ (1+x)^{\gamma}+(1-x)^{\gamma} \}=2, \ \forall \gamma \in (0,1).
\end{eqnarray*}
\end{proof}

\begin{remark}
Theorem \ref{thm3} and Theorem \ref{thm5} state that if the (k+1)-th derivative of $(u-u_N)$ converges, then the $\alpha$-th derivative of $(u-u_N)$ 
must converge with at least the same convergence rate.
\end{remark}

The superconvergence points are located by the following theorem.

\begin{theorem}
Let $\alpha \in (k,k+1)$, let $u_{NB}$ be the interpolant defined in (\ref{def33}). The $\alpha$-th Riesz fractional derivative superconverges at $\{ _B\xi_i^{\alpha} \}$, which satisfies
\begin{eqnarray}
^RD^{\alpha} \mathcal{J}^{-(k+1),-(k+1)}_{N+1}(_B\xi_i^{\alpha})=0,\ i=1,\cdots,N-1. 
\label{superconriz}
\end{eqnarray}
\label{thm7}
\end{theorem}
\begin{proof}
The framework of the proof is the same as Theorem \ref{thm6}. The detail is referred to the Theorem 3.1 in \cite{Deng+arxiv}.
\end{proof}

\section{Numerical Examples}
We use examples 1-3 to investigate the superconvergence points for interger-order derivatives, Riemann fractional derivatives, and Riesz fractional derivatives, respectively.  After that, we apply the theoretical results to solve the integer-order/fractional boundary value problems to illustrate the effectiveness of the theoretical analysis and application of superconvergence points.

\subsection{Superconvergence in derivatives of interpolations}
For the convenience of notations, we define $e_N(x):=u(x)-u_N(x)$. \\

\noindent {\bf Example 1.} Set $f(x)=\frac{1}{1+4x^2}$, which is an analytic function with two simple poles $z=\pm \frac{i}{2}$ in the complex plane. Take the degree of interpolation polynomial $N=31$.  Noting that $f(x)$ is even, we apply two-point Hermite interpolations to demonstrate the superconvergence points as shown in Theorem \ref{thm2} with the following two cases.

\begin{figure}[htbp] 
   \centering
    \includegraphics[width=2.5in]{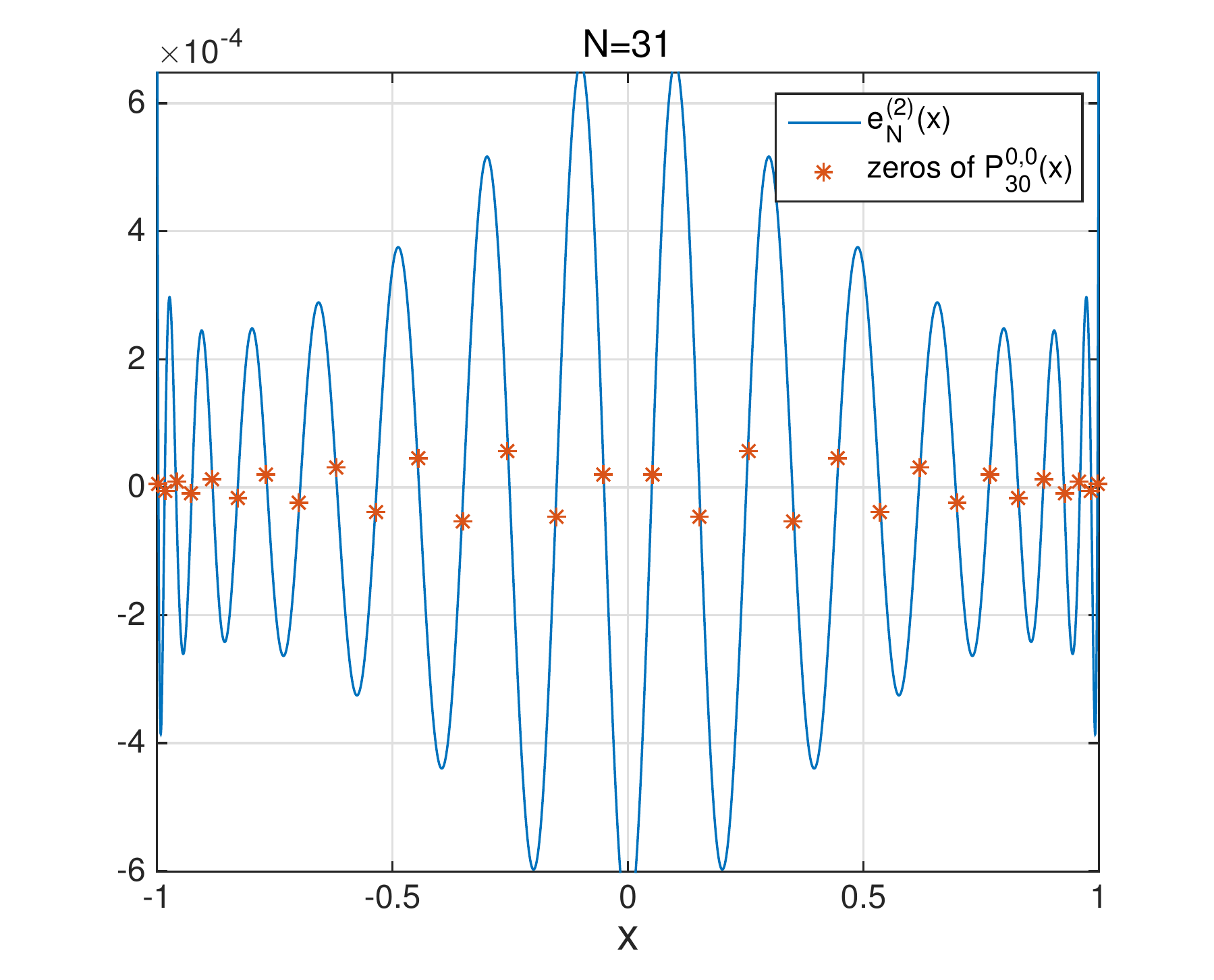}
     \includegraphics[width=2.5in]{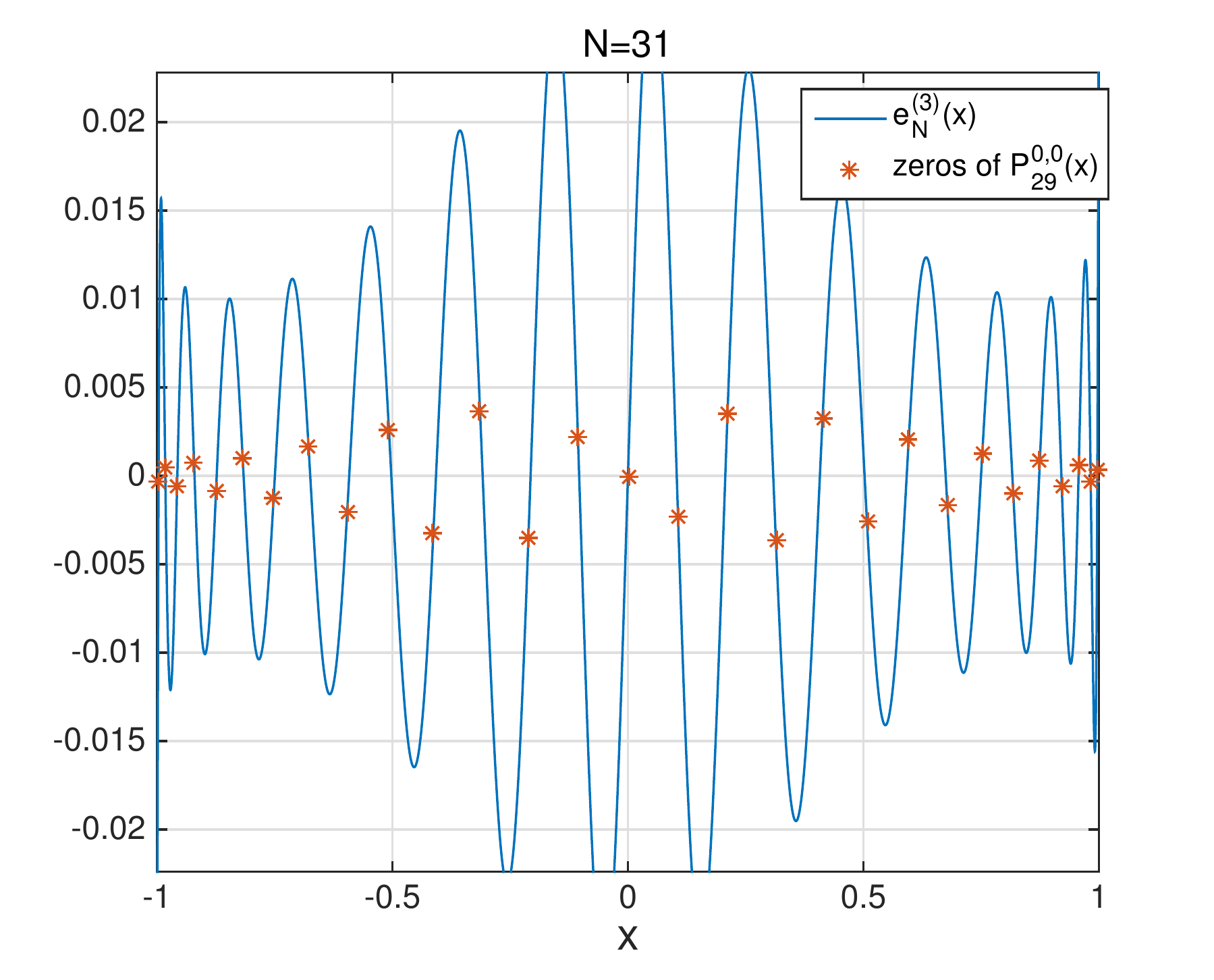}
   \caption{(Example 1) Left: the second derivative of error $e_N^{(2)}(x)$ with $N=31$ and the zeros of $P_{30}^{0,0}(x)$. Right: the third derivative of error $e_N^{(3)}(x)$ with $N=31$ and zeros of $P_{29}^{0,0}(x)$.}
   \label{fg61}
\end{figure}
\noindent {\bf Case 1:}  Assume that the errors satisfy  $e_N(-1)=e_N(1)=e_N'(-1)=e_N'(1)=0$. Left panel of Fig. \ref{fg61} shows that the second derivative of $e_N^{''}(x)$, and  the superconvergence points predicted as the zeros of $P_{30}^{0,0}(x)$ based on Theorem \ref{thm2}.

\noindent {\bf Case 2:} Assume that the errors satisfy $e_N(-1)=e_N(1)=e_N'(-1)=e_N''(1)=e_N''(1)=e_N''(-1)=0$. Right panel of Fig. \ref{fg61} shows the third derivative of $e_N^{'''}(x)$, and  the superconvergence points predicted as the zeros of $P_{29}^{0,0}(x)$ based on Theorem \ref{thm2}. \\

\noindent {\bf Example 2.} This example is used to observe the superconvergence phenomenon for Riemann-Liouville fractional derivatives by taking $f(x)=\frac{1}{10}(1+x)^{6.15}$, $N=18$, and $\alpha=1.23, 1.78, 2.23, 2.78$ respectively. Two cases are also considered.

\noindent {\bf Case  1 ($1<\alpha<2$):} In this situation, we set $e_N(-1)=e_N'(-1)=0$ by adopting the interpolation (\ref{def31}).  Left panel of Fig.  \ref{fg63} shows the errors ${_{-1}D_x^{\alpha}}e_N(x)$ with $\alpha=1.23, 1.78$, and the superconvergence points predicted as zeros of $_{-1} I_x^{2-\alpha} P^{2,0}_{17}(x)$ based on Theorem \ref{thm6}.

\noindent {\bf Case  2 ($2<\alpha<3$):} In this situation, we set  $e_N(-1)=e_N'(-1)=e_N''(-1)=0$ by adopting the interpolation (\ref{def31}) again. Right panel of Fig. \ref{fg63} shows the errors ${_{-1}D_x^{\alpha}}e_N(x)$ with $\alpha=2.23, 2.78$, and the superconvergence points predicted as zeros of $_{-1} I_x^{3-\alpha} P^{3,0}_{16}(x)$ based on Theorem \ref{thm6}.

\begin{figure}[htbp] 
   \centering
    \includegraphics[width=2.5in]{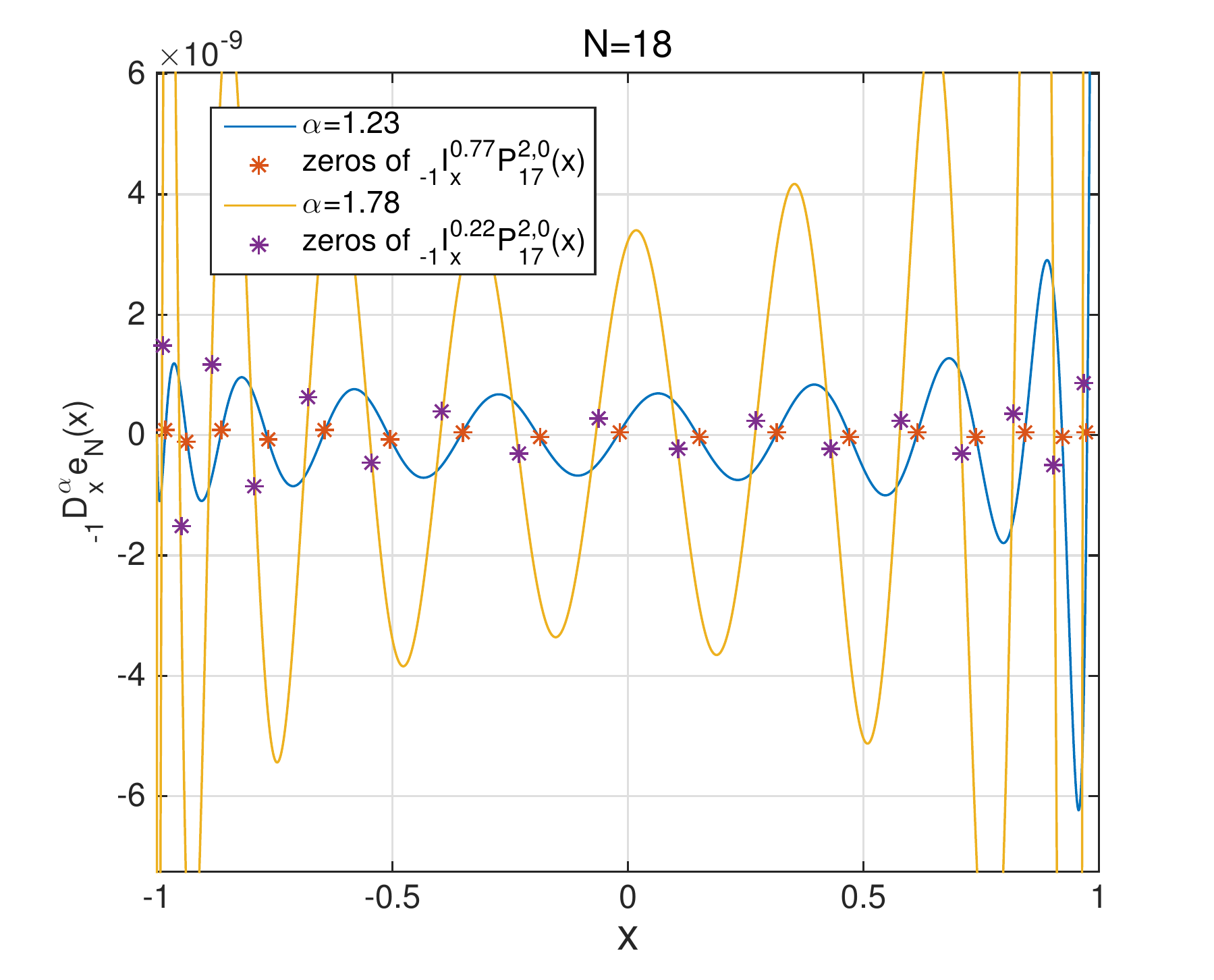}
      \includegraphics[width=2.5in]{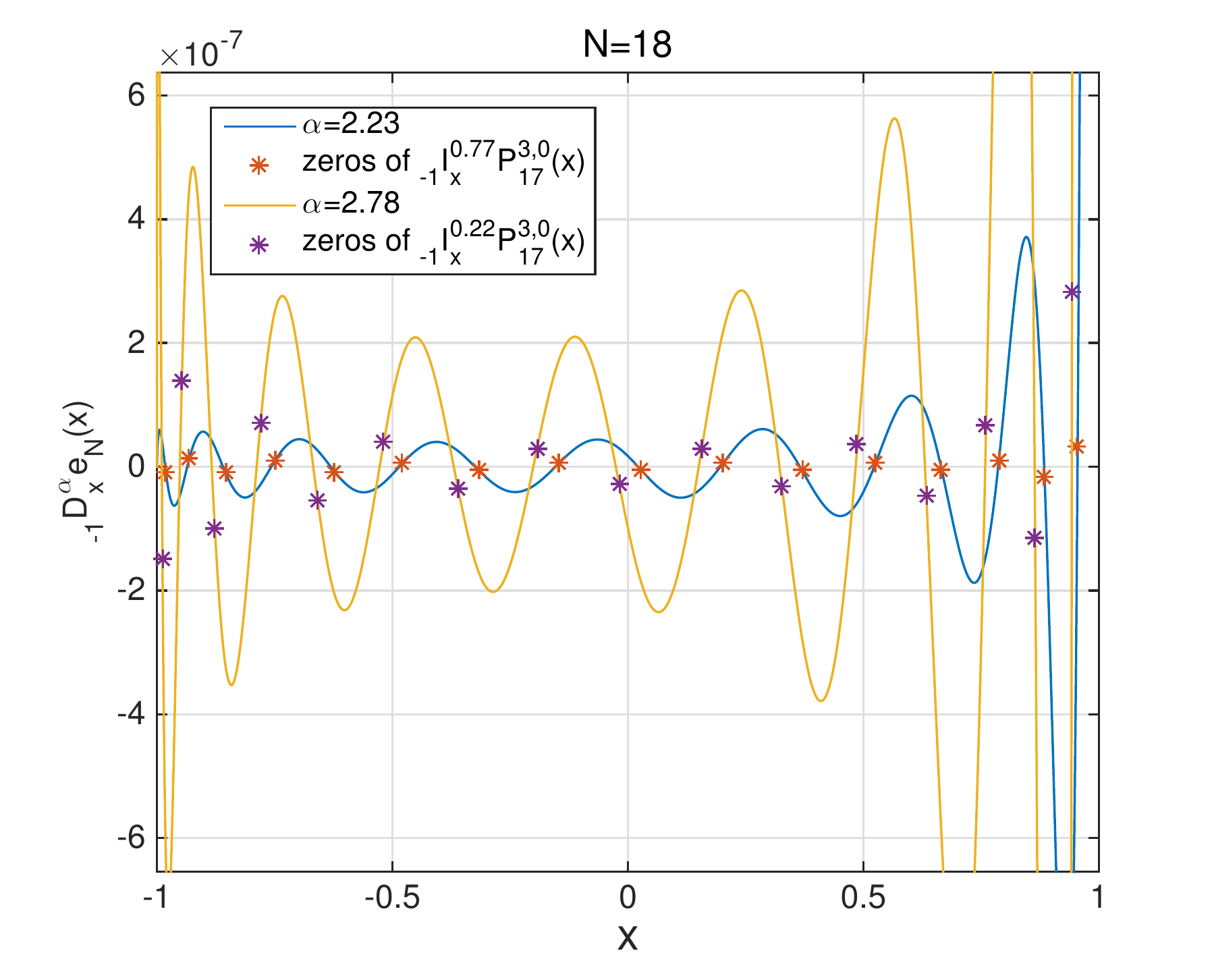}
   \caption{(Example 2) Left: Riemann-Liouville fractional derivative errors ${_{-1}D_x^{\alpha}}e_N(x)$, and the superconvergence points predicted as zeros of $_{-1} I_x^{2-\alpha} P^{2,0}_{17}(x)$ with $N=18$ and $\alpha=1.23, 1.78$. Right: the errors ${_{-1}D_x^{\alpha}}e_N(x)$, and the superconvergence points predicted as zeros of $_{-1} I_x^{3-\alpha} P^{3,0}_{16}(x)$ with $N=18$ and $\alpha=2.23, 2.78$.}
   \label{fg63}
\end{figure}

To further quantify the superconvergence rate, we define the superconvergence ratio as
\begin{eqnarray}
r(N)=\frac{\max_{-1\leqslant x \leqslant 1}{\vert ^R D^{\alpha}(u-u_N)(x) \vert}}{\max_{0\leqslant i \leqslant n}{\vert ^R D^{\alpha}(u-u_N)(\xi_i^{\alpha}) \vert}}.
\label{ratio1}
\end{eqnarray}
According to (\ref{defspcvg}), $r(N)$ is exactly $N^{\beta}$. For any fixed $\alpha,$ we can evaluate the ratio for different $N$ values, and then estimate the gain of rate $\beta$  numerically. As what one can see in Table 6.1, the convergence rates at the superconvergence points are at least $O(N^{-1})$, for different $\alpha$, higher than the global rate, and the gain of convergence rate increases while the derivative order is increasing. \\

\begin{table}[htbp]
   \centering
   \begin{tabular}{@{} cc|cc @{}} 
            \multicolumn{4}{c}{Superconvergence ratios of (\ref{ratio1}) for different $\alpha$th-derivatives.} \\
            \hline
     derivative order ($\alpha$)    & gain of rate ($\beta$) & derivative order ($\alpha$) & gain of rate ($\beta$) \\
     \hline
         1.13         & 1.43     &  2.13 & 1.65\\
     1.33      & 1.51 &  2.33 & 1.66\\
      1.53      & 1.47  & 2.53 & 1.68\\
      1.73 & 1.52   & 2.73 & 1.71\\
      1.93 & 1.52   & 2.93 & 1.73\\
      \hline
    \end{tabular}
         \caption{}
   \label{tab1}
\end{table}

\noindent {\bf Example 3.} This example is used to investigate the superconvergence for Riesz fractional derivatives by taking  $f(x)=(1+x)^9(1-x)^9$, and $\alpha=1.62, 2.44$. The function is interpolated at zeros of $\mathcal{J}^{-2,-2}_{18}(x)$ for $\alpha=1.62$, and at zeros of $\mathcal{J}^{-3,-3}_{18}(x)$ for $\alpha=2.44$ respectively. Similar to previous examples, the errors of ${_{-1}D_x^{\alpha}}e_N(x)$ are shown, and the superconvergence points are highlighted in the Fig. \ref{fg65} respectively.

\begin{figure}[htbp] 
   \centering
   \includegraphics[width=2.5in]{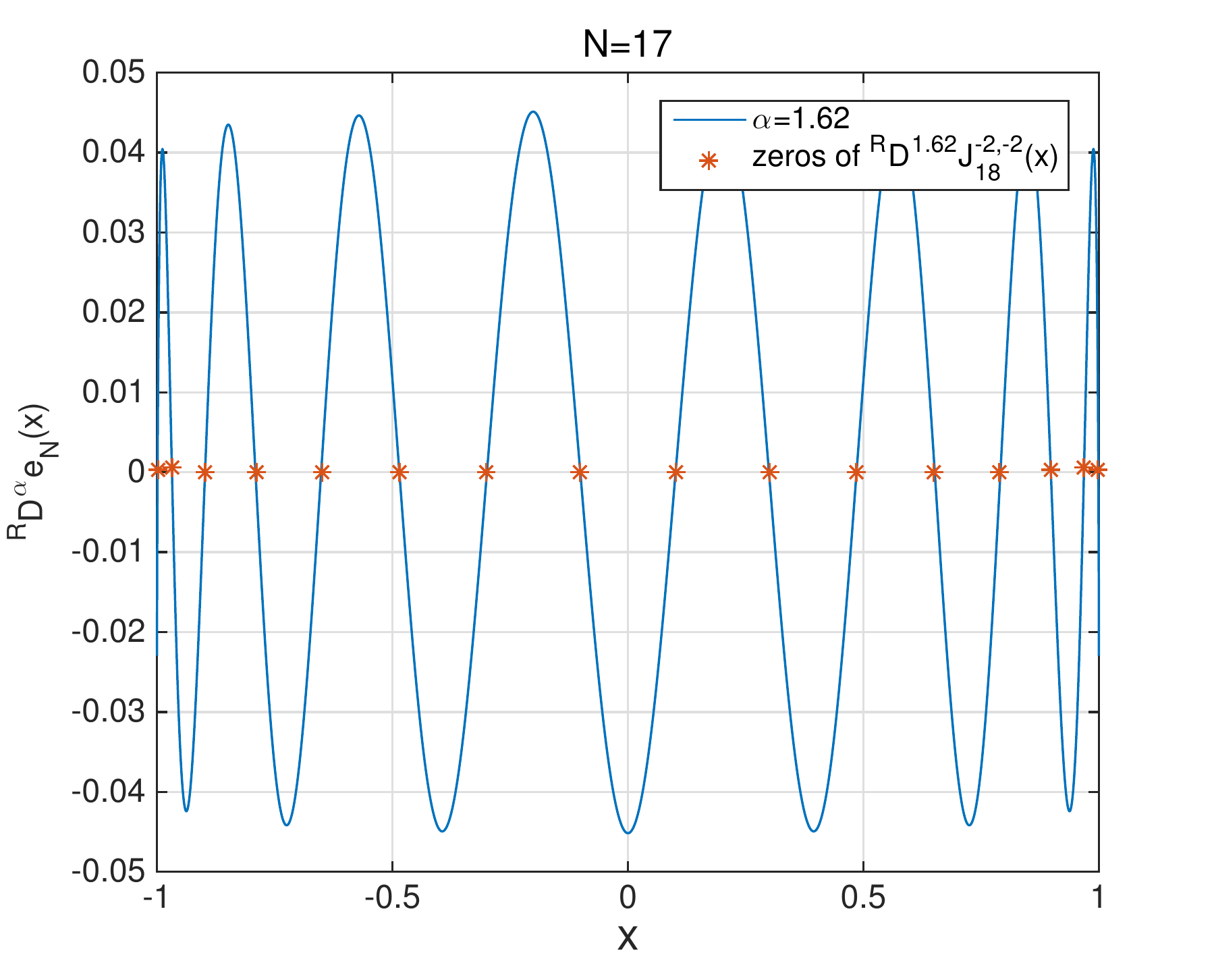}
    \includegraphics[width=2.5in]{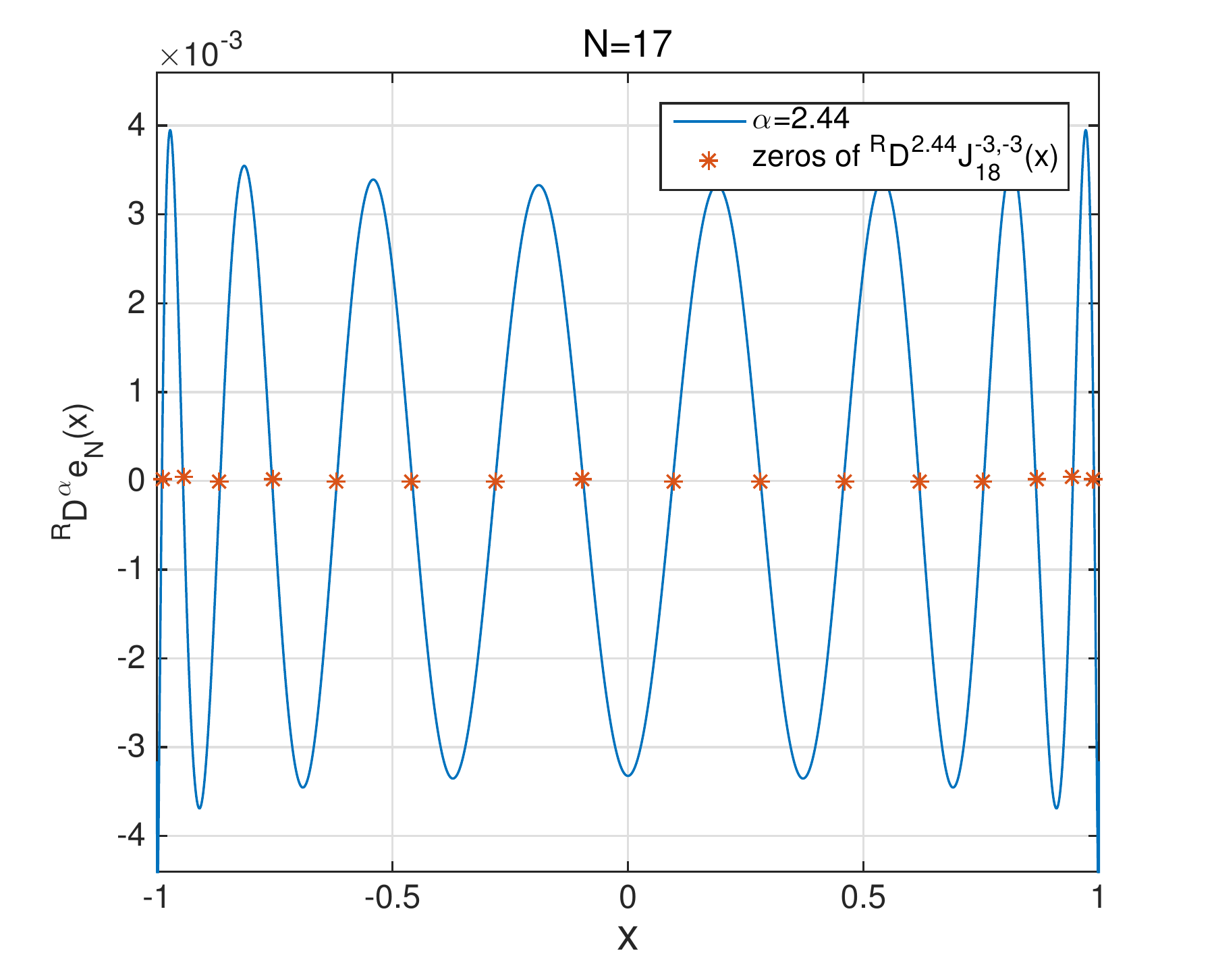}
   \caption{(Example 3) Left: Riesz fractional derivative errors ${^RD^{\alpha}}e_N(x)$ with $\alpha=1.62$. Right: the errors ${^RD^{\alpha}}e_N(x)$ with $\alpha=2.44$. }
   \label{fg65}
\end{figure}
%

In all simulations above, we can observe that the error at highlighted points is much less than the global error.
\\

\subsection{Applications}
We now apply superconvergence points above to solve integer-order/fractional boundary value problems, and investigate the effectiveness of our theoretical analysis. \\

\noindent {\bf Example 4.} Consider the following fourth-order boundary value problem:
\begin{eqnarray}
\left\{\begin{array}{ll}
u^{(4)}(x)+u(x)=f(x), \  x\in(-1,1), \\
u(-1)=u'(-1)=u(1)=u'(1)=0.
\end{array}\right.
\end{eqnarray}
To illustrate the quantitative pictures to the behavior of numerical solutions, we take the exact solution in the form of $u(x)=a_0+30e^x+a_2e^{1.5x}+a_3 e^{2x}+a_4 e^{2.5x}$, where  the coefficients $a_0,a_2,a_3,a_4$ and $f(x)$ can be calculated by using boundary conditions and the exact solution itself. Numerical solutions are obtained by the spectral-Galerkin method, see the detailed numerical scheme provided in (\cite{shen+wang+tang+spectral}, Chap. 6). Left panel of Fig.  \ref{fg67} shows the error of the first derivative $(u-U_N)'(x)$, and the first derivative superconverge points at the interior zeros of $\mathcal{J}^{-1,-1}_N(x)$ based on the analysis of Theorem \ref{thm2},. Similarly, right panel of Fig. \ref{fg67} shows the error of $(u-U_N)''(x)$,  and the second derivative superconvergens at Guass points. \\

\begin{figure}[htbp] 
   \centering
    \includegraphics[width=2.5in]{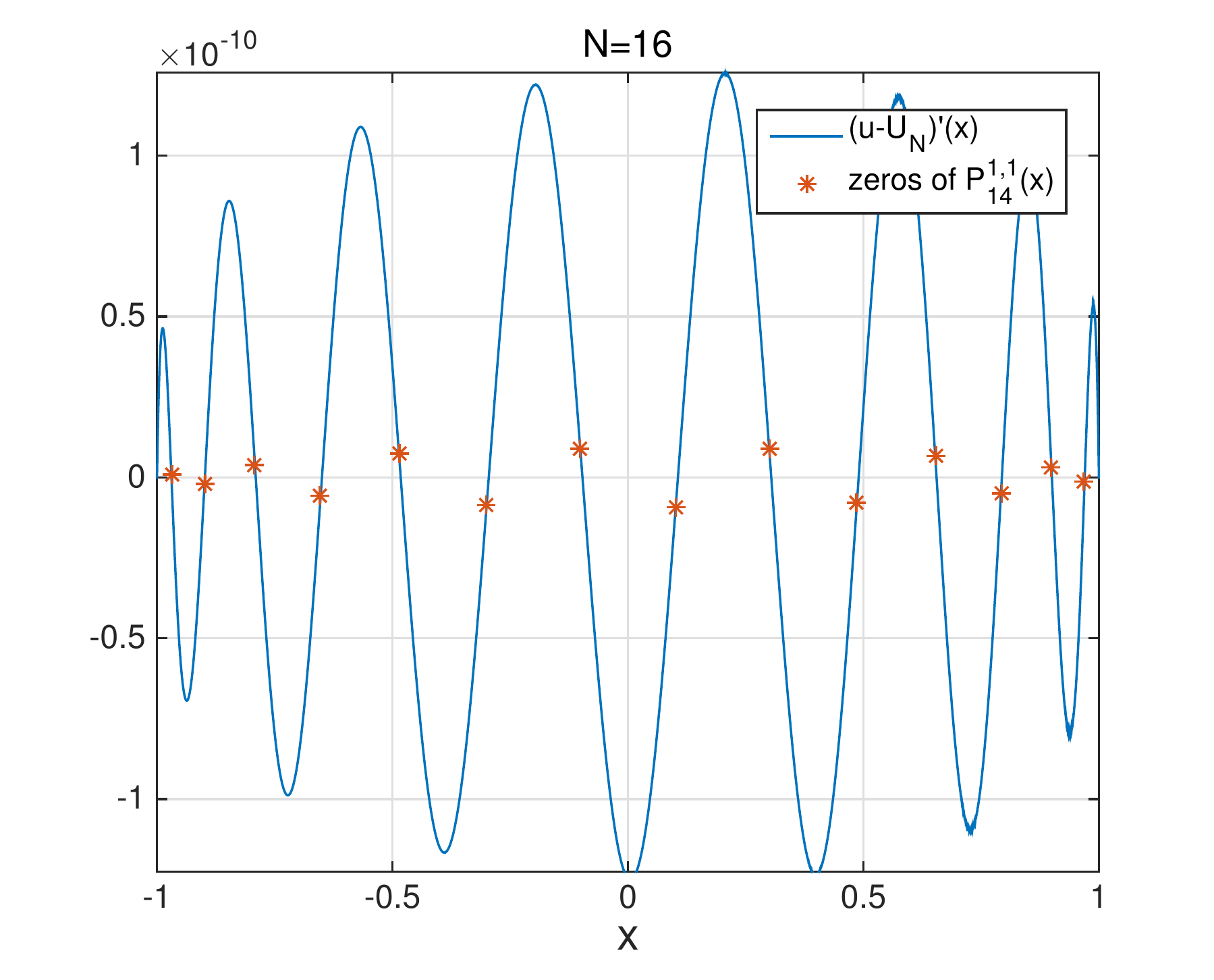}
     \includegraphics[width=2.5in]{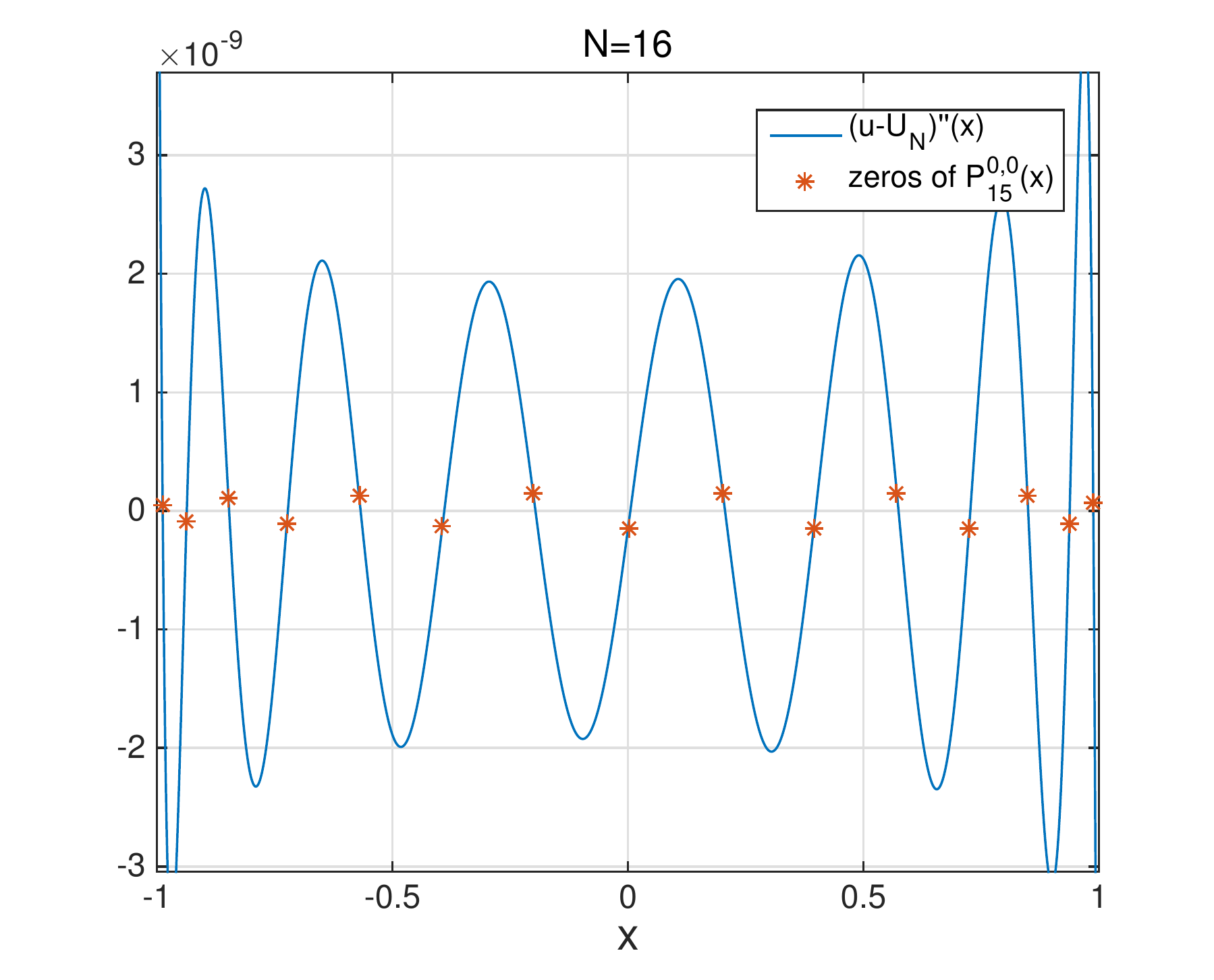}
   \caption{Left: the first derivative error: $e_N'(x)$, with $N=16$; the zeros of $P^{1,1}_{14}(x)$ are asterisked. Right: the second derivative error: $e_N''(x)$, with $N=16$; the Gauss points are asterisked.}
   \label{fg67}
\end{figure}

\noindent {\bf Example 5.} Consider the fractional boundary value problem:
\begin{eqnarray}
\left\{\begin{array}{ll}
{_{-1}D_x^{\alpha}}u(x)=f(x), \  x\in(-1,1),\quad 1<\alpha<2,\\
u(-1)=u'(-1)=0.
\end{array}\right.
\label{example5}
\end{eqnarray}

Again, to illustrate the quantitative pictures to the behavior of numerical solutions, we set the exact solution  $u(x)=\frac{1}{10}(1+x)^{6.15}$, and $f(x)$ can be calculated correspondingly. The problem \eqref{example5} is solved by spectral-collocation methods. According to the original collocation method, let $\{x_i \}_{i=1}^{10}$ be the interior zeros of $\mathcal{J}_{12}^{0,-2}(x)$. The numerical solution can be written in the form of
\begin{eqnarray}
U_N^o(x)=\sum_{j=1}^{N-1}  {U_j^o} \hat{\ell}_j(x):=\sum_{j=1}^{N-1}  {U_j^o} \cdot \ell_j(x) \frac{(1+x)^2}{(1+x_j)^2},
\end{eqnarray}
where $\ell_j \in \mathbb{P}_N[-1,1]$ is the Lagrange basis function satisfying
$${\ell}_j(x_i)=\delta_{ij}, \ i,j=1,\ldots,N-1.$$
Hence, the numerical solution is to find $ {\bf{U}}^o=(U_1^o,U_2^o, \cdots, U_{N-1}^o)^T \in \mathbb{R}^{N-1}$ such that
\begin{eqnarray}
(_{-1}D_x^{\alpha}U_N^o)(x_i)=f(x_i), \quad { i=1,\ldots,N-1.}
\label{example6231}
\end{eqnarray}
In the simulation, we set $\alpha=1.31$, $N=11$. Left panel of Fig. \ref{appl61} shows the error $(u-U_N^o)(x)$, and right panel of Fig. \ref{appl61} shows the fractional derivative of error ${_{-1}D_x^{\alpha}}(u-U_N^o)(x)$ and the corresponding interpolation points.  From Fig. \ref{appl61}, one can see that that ${_{-1}D_x^{\alpha}}(u-U_N^o)(x)$ vanishes at $\{x_i \}_{i=1}^{10}$, but the error $(u-U_N^o)(x)$ doesn't have superconvergence phenomenon, since it doesn't oscillate around 0.

Based on the theoretical analysis in this paper, we may modify the numerical scheme (\ref{example6231}) as: find $ {\bf{U}}^n=(U_1^n,U_2^n, \cdots, U_{N-1}^n)^T \in \mathbb{R}^{N-1}$ such that
\begin{eqnarray} \label{newm}
(_{-1}D_x^{\alpha}U_N^n)(\xi_i)=f(\xi_i), \quad { i=1,\ldots,N-1,}
\end{eqnarray}
where $\{\xi_i \}_{i=1}^{10}$,  the zeros of $_{-1} I_x^{2-\alpha} P^{2,0}_{10}(x)$, are the superconvergence points predicted in Theorem \ref{thm6}. As proved above, Theorem \ref{thm4} guarantees that the number of the superconvergence points $\{\xi_i \}$ is as same as of interpolation points $\{x_i \}$, so that the linear system has unique solution.

As what shown in Fig \ref{appl62}, one can observe that ${_{-1}D_x^{\alpha}}(u-U_N^n)(x)$ vanishes at $\{\xi_i \}_{i=1}^{10}$. More importantly, $U_N^n(x)$ superconverges to $u(x)$ at the interpolation points $\{x_i \}_{i=1}^{10}$. This leads that numerical solution of our new scheme \eqref{newm} is much more accurate than that of the traditional scheme \eqref{example6231}. To illustrate the accuracy of numerical schemes, for any given degree of freedom $N$, we define the maximum norms of numerical solutions respectively for schemes \eqref{example6231} and \eqref{newm} as
\begin{eqnarray*}
E_o=\max_{1 \leqslant i \leqslant N-1} \{ |u(x_i)-U_i^o |\} , \ \hbox{and} \ E_n=\max_{1 \leqslant i \leqslant N-1} \{ |u(x_i)-U_i^n| \}.
\end{eqnarray*}
The maximal error norms $E_o$ and $E_n$ are list Table (\ref{table2}). One can see that for each $N$, $E_n$ is significantly smaller than $E_o$ at least by $10^{-2}$.

\begin{figure}[htbp] 
   \centering
    \includegraphics[width=5in]{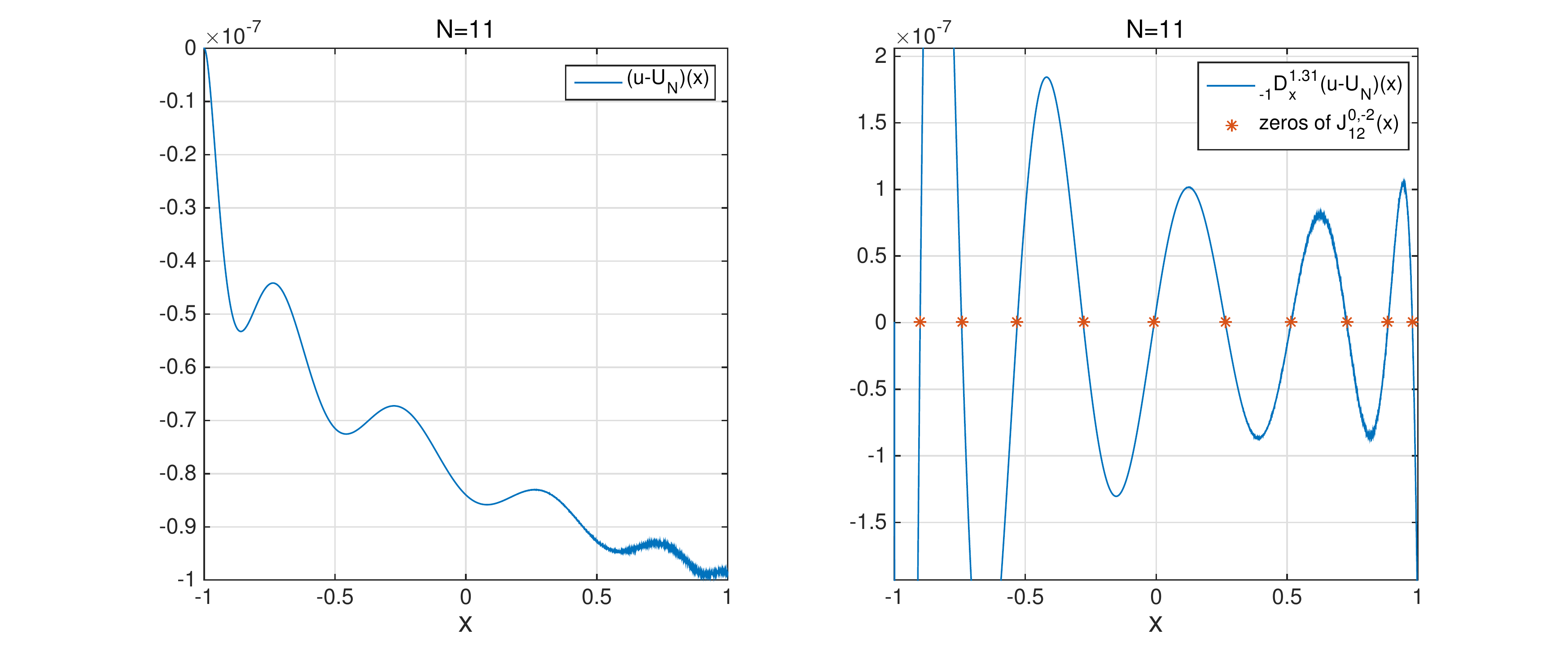}
   \caption{Right: Fractional derivative errors: ${_{-1}D_x^{1.31}}e_N(x)$, with $N=11$, evaluating points $\{x_i \}_{i=1}^{10}$ (asterisked). Left: the error $e_N(x)$, the curve doesn't oscillate around 0.}
   \label{appl61}
\end{figure}

\begin{figure}[htbp] 
   \centering
    \includegraphics[width=5in]{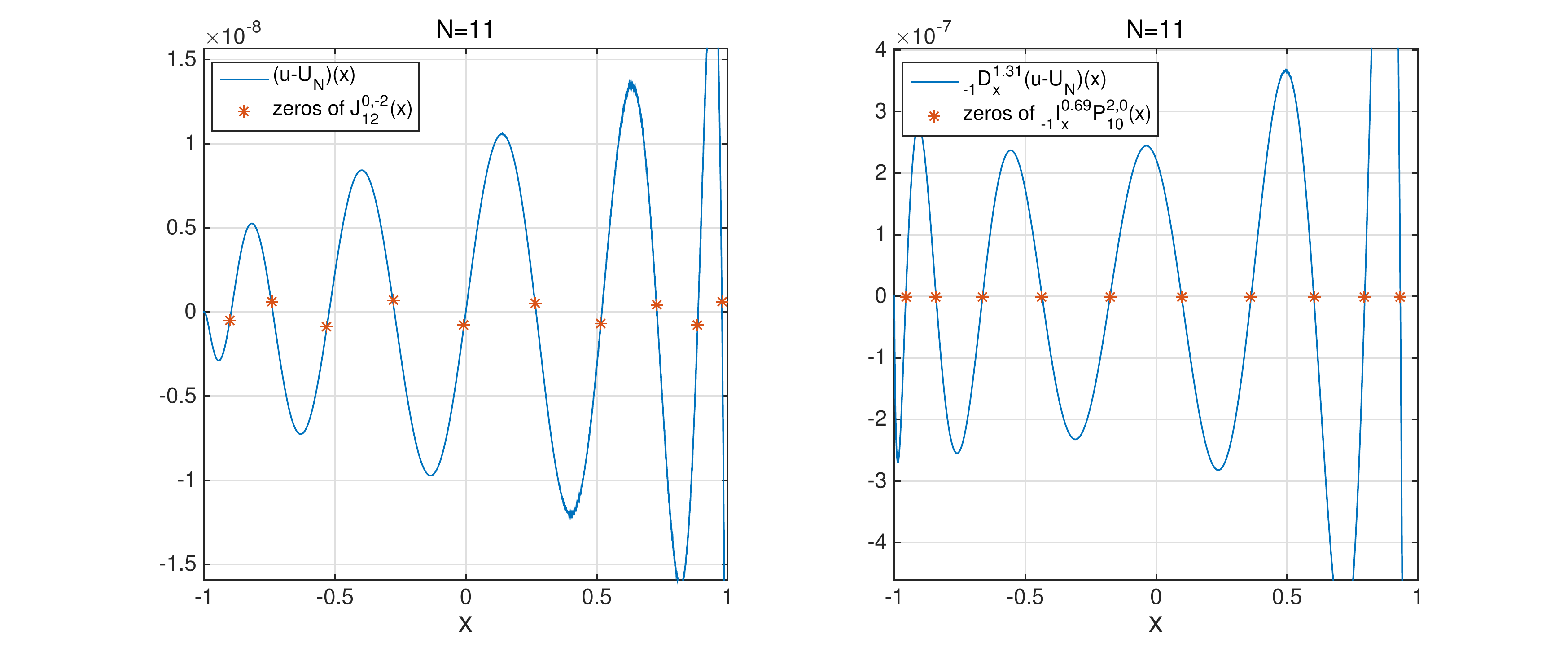}
   \caption{Right: Fractional derivative errors: ${_{-1}D_x^{1.31}}e_N(x)$, with $N=11$, evaluating points $\{\xi_i \}_{i=1}^{10}$ (asterisked). Left: the error $e_N(x)$, the curve oscillates around 0, and the error at the asterisked points (interpolation points) is much less than the global error.}
   \label{appl62}
\end{figure}

\begin{table}[htbp]
	\centering
	\begin{tabular}{@{} c|c|c|c|c|c|c @{}}
	\multicolumn{7}{c}{The comparison of maximal errors of the two collocation methods} \\
	\hline
	N & 6 & 7 & 8 & 9 & 10 & 11\\
	\hline
	$E_o$ & 6.45E-04 & 4.52E-05 & 6.46E-06 & 1.29E-06 & 3.32E-07 & 9.84E-08 \\
	\hline
	$E_n$ & 3.42E-06 & 3.52E-07 & 5.54E-08 & 1.15E-08 & 2.90E-09 & 8.55E-10 \\
	\hline
	\end{tabular}
	\caption{}
	\label{table2}
\end{table}

\section{Conclusions}
Generally, the three kinds of Hermite interpolations are comparable to Lagrange-type spectral interpolations. It is proved that for any fixed $k$, the integer-order derivative of the error decays exponentially with respect to the degree of freedom.
The superconvergence points of each Hermite interpolation are also identified. Moreover, the convergence rates at the superconvergence points of one-point and two-point Hermite interpolations are proved to be $O(N^{-2})$ and $O(N^{-\frac{3}{2}})$ higher than the global rate respectively.

The superconvergence points of the Riemann-Liouville derivative can be efficiently and accurately calculated by an eigenvalue method. Numerical simulations show that the gain of convergence rate is at least $O(N^{-1})$. It is a future work to validate the gain of convergence rate for fractional derivatives theoretically. Another interesting observation is that the number of interpolation points for the Riemann-Liouville fractional derivative  is equal to the number of superconvergence points, which differs from the case of integer-order derivative. Therefore, we can use the
superconvergence points found in this paper to numerically solve FDEs.  By modifying the collocation scheme based on the analysis in Theorem \ref{thm6}, the errors at interpolation points are much smaller than the global errors. This phenomenon is not observed in the numerical solution of traditional collocation method. So comparing with the traditional collocation method,
our new scheme produces more accurate numerical solutions.

As far as we know, this is the first attempt to study the superconvergence points for Riemann-Liouville and Riesz fractional derivatives by using Hermit interpolation. As shown in \cite{Mao+2015+ANM,Li+2012+FCAA,Li+2009+SINUM,Xu+2014+JCP,Zeng+2014+SINUM,Zheng+2015+SISC,Can+2016+SINUM,Can+2018+JSC,Zeng+Gg+2015+SISC,Chen+2014+mathcom}, the spectral method has been successfully applied to solve FDEs. The study in this paper will be useful to numerally solve FDEs, especially, using spectral collocation methods. Hence, it will be interesting to develop new spectral collocation methods based on the superconvergence points for high-order FDEs, such as super-diffusion models in \cite{MMMeer+2007,CLi+YQChen+2011}, or FDEs with Riesz, Caputo, or two-sided fractional derivatives in \cite{Zeng+2014+SINUM, Deng+2012+AM,Shen+2014+IMA, XuanSun+2014+SISC,Martin+2014+finitedf, Martin+2015+colloc, Mao+2018+SINUM}.

\end{document}